\documentclass{article}
\usepackage{graphicx} 
\usepackage[a4paper, total={6in, 8in}]{geometry}
\usepackage{amsmath, amsfonts, amsthm}
\usepackage{booktabs}
\usepackage{color}
\usepackage{xcolor}
\usepackage{enumitem} 
\usepackage{microtype}
\usepackage{float}

\allowdisplaybreaks 
\def\R{\mathbb{R}}

\def\cE{\mathcal{E}}

\def\m{\mathfrak{m}}

\newcommand{\E}{\mathbb{E}}

\newcommand*\der{\mathop{}\!\mathrm{d}}

\newtheorem{theorem}{Theorem}[section]
\newtheorem{lemma}[theorem]{Lemma}

\newtheorem{remark}{Remark}[section]

\numberwithin{equation}{section}
\newtheorem{assumption}{Assumption}

\usepackage{mathrsfs}
\providecommand{\keywords}[1]{%
  {\small\noindent\textbf{\textit{Key words.}}\ #1}%
}

\providecommand{\msc}[1]{%
  {\small\noindent\textbf{\textit{AMS subject classifications.}}\ #1}%
}

\title{Treasure Search Optimization}
\author{A. Sharma}
\date{}

\begin{document}

\maketitle
\begin{abstract}

We introduce Treasure Search Optimization (TSO), an interacting particle method for global optimization. Most swarm methods balance exploration and exploitation within a single population, and typically switch between the two by degenerating the noise, annealing a temperature, or tuning a parameter. TSO instead splits these tasks across two kinds of agents. A swarm of explorers stays in exploration mode and a single treasure hunter performs exploitation. The hunter drifts toward an objective-weighted average of the explorers and may teleport to it when the move lowers the objective. The swarm then re-centers around the hunter, creating a feedback loop between search and capture. We model the dynamics as coupled jump-diffusion stochastic differential equations (SDEs). The hunter's jumps are shared by all explorers and act as a common noise. The mean-field limit is therefore a conditional McKean-Vlasov jump-diffusion SDE, whose well-posedness we prove. We also characterize the steady state and prove, via Laplace approximation techniques, that the hunter settles near the global minimum with error of order $1/\alpha$, where $\alpha$ is the weight parameter. Linking the consensus drift to a smoothed free energy, we explain why the swarm ignores spurious local traps and demonstrate how to quantify uncertainty in inverse problems using post-processing Kalman steps after TSO iterations. Numerical experiments on ODE-constrained problems and a low dimensional Bayesian inverse problem demonstrate the effectiveness of the TSO method.


\msc{90C26, 60H10, 65K10, 65C30, 65C35, 60J76
}

\end{abstract}

\section{Introduction}

Interacting particle systems (IPS) provide a flexible modeling framework where  evolution of large populations of agents is governed by deterministic forces, random fluctuations and interaction forces such as attraction, repulsion, synchronization, competition, or cooperation. They arise naturally in social and natural sciences, engineering, and machine learning, with applications ranging from collective animal behavior, opinion formation, and crowd dynamics to swarm robotics \cite{albi2019vehicular,SY_Tadmor_flocking,bartel2026multi, bellomo2017quest, motsch_tadmor, ant_colony, handbook_evo_comp}.

\paragraph{IPS based sampling and optimization.}
The recent surge of interest in SDEs governed interacting particle systems (IPS) for optimization and sampling is driven by their ability to overcome the limitations of classical algorithms in anisotropic non-convex landscapes. An isolated particle may get trapped in local minimum, IPS leverage collective dynamics, sharing information across the swarm, to efficiently navigate complex spaces. Specifically, in non-convex optimization, interaction mechanisms such as attraction, consensus, or selection drive the swarm toward promising regions. In sampling, the empirical particle distribution can approximate complex target measures, with suitable interactions such as ensemble covariance preconditioning or repulsion. Crucially, many of these methods do not require gradient information, making them ideal for complex inverse problems. 

Consequently, a wide variety of successful algorithms have emerged. Important examples of IPS based methods for optimization and sampling include  ensemble Kalman inversion optimizer \cite{iglesias2013ensemble, chada2020tikhonov,chada2022convergence}, Stein variational gradient descent \cite{liu2019stein}, Fokker-Planck interacting particle systems \cite{reich2021fokker}, ensemble Kalman-Langevin samplers \cite{garbunoinigo2019interacting,garbuno2020affine,ringh2025kalman},   particle swarm optimization \cite{kennedy_particle_1995} (and its continuous-time version \cite{grassi2023mean}), consensus-based optimization (CBO) \cite{pinnau2017consensus,carrillo2018analytical,totzeck2021trends},   consensus-based sampling \cite{carrillo2022consensus_sampling}, swarm-based annealing \cite{ding2024swarm}, swarm gradient descent \cite{bolte2024swarm}, ensemble-based gradient inference \cite{schillings2023ensemble}, ensemble affine invariant samplers \cite{goodman2010ensemble} and the Boltzmann sampler \cite{chen2024bayesian_boltzmann}.  All of these methods use collective dynamics to balance exploration and exploitation in complex non-convex landscapes.

Several variants and extensions also exist, for example polarized or multi-swarm CBO \cite{bungert2025polarized,klamroth2024consensus}, constrained  methods \cite{borghi2023constrained, bae2022constrained, cbo_eki_constrained, bungert2025mirrorcbo}, ensemble Kalman-Langevin multiscale sampler \cite{pavliotis2022derivative}, models with jumps \cite{kalise2023consensus,borghi2025swarm_bgk},  control strategies for optimization \cite{molin2025controlled,huang2026fast, aubert2026controlled} and consensus-based dynamics in infinite-dimensional setting \cite{huang2026derivative}. Specifically, for constrained settings, several of the above models can be formulated in terms of mean-field reflected SDEs, studied in \cite{adams2022large} for convex domains and in \cite{hinds2025well} for non-convex domains.

While these existing methods successfully leverage collective dynamics, they typically employ a homogeneous swarm. To combine the exploratory capabilities of samplers with the targeted exploitative capabilities of optimizers within a single population one often relies on noise annealing or noise degeneracy. However, this forces a severe trade-off: it either leads to premature collapse (in the case of degenerate noise)  or exponentially slow escape from local minima (in the case of annealing via time-dependent noise coefficient). Motivated by the need to explicitly decouple these conflicting tasks, we introduce a dual mechanism interacting particle system. 

\paragraph{Treasure Search Optimization.} In this paper, we propose a multi-agent stochastic model governed by two coupled interaction mechanisms. The system consists of $N$ `treasure explorers' who constantly remain in exploratory mode forming a swarm around a time-varying center. Even if initialized from the same point, these explorers quickly disperse to form a time-varying cloud. 
The explorers drift, diffuse and jump but according to a particular mechanism designed to drive the system toward a steady state. If one of the explorers makes a rare discovery, the  signal is amplified, which is captured by the `treasure hunter' with a positive probability through a jumping mechanism. Once the jump is realized, the swarm of explorers begins re-centering around the treasure hunter in the landscape, and therefore establishing a feedback loop. This results in a continuous interplay between the wider exploration by the swarm and the targeted exploitation by the hunter. We term the above described dynamics Treasure Search Optimization (TSO).

The central idea of Treasure Search Optimization is therefore architectural rather than tied to a specific interaction force: it prescribes how exploration and exploitation are separated across two distinct types of agents.

\paragraph{Contributions.} The contributions of the paper are as follows.
\begin{itemize}
\item[(i)] We formalize the above proposed hunter-explorer dynamics in terms of jump-diffusion SDE. We then contrast its governing dynamics and interaction mechanisms with other multi-agent models.
 Because the explorers are subject to a common noise driven by the Poisson jumps of the treasure hunter, the limiting mean-field dynamics is a conditional McKean-Vlasov jump-diffusion SDE. We establish the well-posedness of this system, which requires certain adaptations beyond the standard Leray-Schauder or Banach fixed-point arguments typically employed in McKean-Vlasov literature.

 \item[(ii)] We characterize the steady state of the coupled dynamics and prove the proximity of the treasure hunter's steady-state position to the global minimum of the objective function.

\item[(iii)] Focusing on the random conditional mean ODE corresponding to the explorers' mean-field dynamics, we investigate the underlying macroscopic gradient structure in several different settings. By formally connecting the consensus drift to the gradient of a ``smoothed" free energy, we characterize the macroscopic behavior of the swarm. This link explains why the collective swarm is effectively  ``blind" to small, spurious local peaks, and thus highlighting the dichotomy between the landscape perceived by the swarm versus individual particles. This insight provides a qualitative lens for analyzing CBO and CBS methods.

\item[(iv)] Finally, we demonstrate the practical efficacy of the proposed method through numerical experiments with ODE-constrained optimization problems and with a Bayesian inverse problem.
\end{itemize}

\paragraph{Organization.}
The remainder of this paper is organized as follows. Section~\ref{tso_section_2} details the SDE formulation of the proposed mechanism and situates it within the broader literature. In the same section, we establish well-posedness of the mean-field SDE. In Section~\ref{tso_steady_state_subsec}, we present the steady-state characterization and our theoretical results regarding closeness to global minimum. Section~\ref{conditional_mean_ode_section} explores the gradient structure of the conditional mean ODE and the free energy framework. Finally, Section~\ref{numerical_exp_sec} provides supporting numerical illustrations.

\paragraph{Notation.} Let $f : \mathbb{R}^{d} \to \mathbb{R}$ be the objective function. For vectors $a, b \in \mathbb{R}^d$, their scalar product is denoted by $(a \cdot b)$ and Euclidean norm by $|a|$. We denote by $\mathcal{P}(\mathbb{R}^{d})$ the space of probability measures on $\mathbb{R}^d$ and with $\mathcal{P}_p(\mathbb{R}^d)$ a subset of it with bounded $p$-th moments. For any two probability measures $\mu, \nu \in \mathcal{P}_2(\mathbb{R}^d)$, we define the $2$-Wasserstein distance as
\begin{align}
    \mathcal{W}_2(\mu, \nu) = \left( \inf_{\pi \in \Pi(\mu, \nu)} \int_{\mathbb{R}^d \times \mathbb{R}^d} |x - y|^2 \pi(\der x, \mathrm{d}y) \right)^{1/2},
\end{align}
where $\Pi(\mu, \nu)$ denotes the set of all couplings (joint distributions) on $\mathbb{R}^d \times \mathbb{R}^d$ with marginals $\mu$ and $\nu$.  For an $\alpha > 0$, we define a weighted average as
\begin{align}
   \m^\alpha(\mu) =  \frac{1}{\int_{\mathbb{R}^d} w^{\alpha, f} \mu(\der x)} \int_{\R^d} x w^{\alpha, f} \mu(\der x), 
\end{align}
where $w^{\alpha, f} : \mathbb{R}^{d} \rightarrow (0,\infty)$ represents the weight function. In this paper, we fix $w^{\alpha , f} = e^{-\alpha f} $. 
Finally, we denote $\|v\|_\infty := \sup_{t \in [0,T]} |v_t| $ for $v \in C([0,T];\mathbb{R}^d)$ where $C([0,T];\mathbb{R}^d)$ represents the space of continuous functions defined on $[0,T]$ and taking values in $\mathbb{R}^d$.

\section{Treasure Search Optimization}\label{tso_section_2}

Let $N \in \mathbb{N}$ denote the number of explorers. Let $X^{i, N}_t\in  \mathbb{R}^{d}$ denote the position of $i$-th explorer among $N$ explorers. Let $Y_t \in \mathbb{R}^{d}$ denote the position of treasure hunter at time $t$.  Also, consider $d$-dimensional Wiener processes $W^{i}$, $i=1,\dots,N$. Also, $N^{i}(\mathrm{d}t,\der z)$, $i=1,\dots,N$, are independent Poisson random measures with intensity
$\lambda_J \nu(\der z)\mathrm{d}t$, where $\nu(\der z)\propto e^{-|z|^2/2}\der z$. Let $N^{Y}_t$ be a Poisson clock with intensity $\lambda_Y$.  Let the empirical measure of these $N$ explorers be denoted by 
\begin{align}
    \cE^{N}(\cdot) := \frac{1}{N}\sum\limits_{i=1}^{N}\delta_{X^{i,N}}(\cdot).
\end{align}
The explorers evolve as
\begin{align}\label{tso_eqn_2.2}
\der X_t^{i,N}
&=
\underbrace{-\eta_1\bigl(X_t^{i,N}-\m^\alpha(\mathcal E_t^N)\bigr) -\eta_2\bigl(X_t^{i,N} - Y_t\bigr)}_{=b(X^{i,N}_t, Y_t, \mathcal{E}_t^N)}\der t
+
\sigma \der W_t^i \nonumber 
\\
&\quad+
\int_{\mathbb R^d}
\Big[
(1-\cos\phi)
\bigl(-X_{t^-}^{i,N}+\mathrm{M}^\alpha(\mathcal E_t^N,Y_t)\bigr)
+
\sigma_J\sin\phi\,z
\Big]
N^i(\der t,\der z),
\end{align}
and the hunter update is given by 
\begin{align}\label{tso_eqn_2.3}
    \der Y_t
    = -\beta\bigl(Y_t - \mathfrak{m}^\alpha(\cE_t^N)\bigr)\der t +
    \bigl(\mathfrak{m}^\alpha(\cE_t^N)-Y_{t^-}\bigr)
    \mathbf 1_{\{f(\mathfrak{m}^\alpha(\mathcal{E}_t^N))<f(Y_{t^-})\}}
    \der N_t^Y,
\end{align} 
where 
\begin{align}\label{tso_eqn_2.4}
\mathrm{M}^\alpha(\mathcal E_t^N,Y_t) =     \frac{\kappa_1 \mathfrak{m}^\alpha(\mathcal E_t^N)+\kappa_2Y_t}{\kappa_1+\kappa_2},
\end{align}
 and $\eta_1, \eta_2, \sigma, \kappa_1, \kappa_2, \beta >0$ and  $\phi\in(0,\pi/2)$.

We use the notation
$b(X_t^{i,N},Y_t,\mathcal E_t^N)$ to emphasize that the deterministic explorer interaction need not be restricted to the
particular consensus-hunter drift displayed in \eqref{tso_eqn_2.2}. Other gradient-free interaction fields depending on the explorer position,
the hunter, and the empirical ensemble can be incorporated, either in
addition to or in place of the consensus-hunter drift. Such choices may also alter
the effective noise structure of the interacting particle system. 

For example, a repulsive variant (see \cite{hinds2025well}), in the spirit of Stein variational dynamics, augments the consensus-hunter pull with a kernel-based repulsion that discourages the explorers from crowding,
\begin{align}\label{tso_eqn_b_stein}
    b = -\eta_1\bigl(X_t^{i,N} - \m^\alpha\bigr) - \eta_2\bigl(X_t^{i,N} - Y_t\bigr) + \frac{\eta_3}{N}\sum_{j=1}^N \nabla_2 K\bigl(X_t^{i,N}, X_t^{j,N}\bigr),
\end{align}
where $K$ is a smooth kernel and $\nabla_2$ denotes the gradient in the second argument.
Other possible choices include ensemble cross-covariance or covariance interactions in the spirit of ensemble Kalman methods, or polarized interactions.  This emphasizes that the general TSO structure of splitting the search by explorers and exploitation of searched space by hunters is not specific to a particular choice of explorer interaction.

We refer to the dynamics described by \eqref{tso_eqn_2.2}-\eqref{tso_eqn_2.4} as the \textit{Treasure Search Optimization} method, or $\mathrm{TSO}$ for short. Although the explorer interaction can be modified within this framework,
we focus in the remainder of the paper on the particular choice specified in
\eqref{tso_eqn_2.2}.
\vspace{-10pt}
 \paragraph{Mechanism of treasure explorers' dynamics:}
 In \eqref{tso_eqn_2.2}, $\eta_1$ and $\eta_2$ respectively control the attraction to the weighted average $\mathfrak{m}^{\alpha}(\mathcal{E}^N_t)$ and the hunter $Y_t$, with continuous diffusive exploration driven by $\sigma>0$. 
 The jump update combines directed attraction and random exploration: the factor $\cos(\phi)$ retains the explorer's pre-jump position, the term $(1- \cos(\phi))$ pulls toward augmented target point $\mathrm{M}^\alpha$, while $\sigma_J \sin(\phi) z$ introduces random exploratory jumps.   
 \vspace{-10pt}
 \paragraph{Mechanism of treasure hunter's dynamics:}
Between consecutive jump times, the hunter evolves by following the current weighted average. At each Poisson jump, the hunter updates its position toward the consensus only if this move leads to an improvement in the objective value. Once the hunter reaches a promising basin, its position contributes to the average $\mathrm{M}^\alpha$, thereby influencing all explorers.

\paragraph{Mechanism of coupled dynamics:} Therefore, the coupled TSO SDEs \eqref{tso_eqn_2.2}-\eqref{tso_eqn_2.3} describe a collective search process driven by four interconnected phenomena:
\begin{align}
\text{exploration}
\;\Longrightarrow\;
\text{amplification}
\;\Longrightarrow\;
\text{capture}
\;\Longrightarrow\;
\text{feedback}.
\end{align}

 First, the persistent jump-diffusions in the dynamics allow the explorers to sample distant basins. When an explorer happens to land in a region with a lower objective value, the exponential Gibbs weighting amplifies this rare event. The treasure hunter with positive probability captures this rare finding only when it is an improvement. Once updated, the hunter's continuous feedback  pulls the entire cloud of explorers toward the newly discovered basin.

By contrast, an additive noise consensus-based model \cite{additive_CBO_2026}(see also \cite{choi2025modified}) would retain exploration and Gibbs amplification, but without the treasure hunter, it entirely lacks the mechanisms for capturing optimal discoveries and providing feedback to the swarm. If we attempt to remove the noise upon a favorable discovery to induce the event of  ``capture", such an intervention would revert the system to vanilla CBO \cite{pinnau2017consensus, carrillo2018analytical}. Isotropic or anisotropic CBO \cite{carrillo2018analytical, carrillo_shi_cbo} has fast ``local" convergence when initial distribution  has sufficient mass ``close" to global minimum. However, due to degeneracy in noise, the exploratory capabilities are limited if initial distribution does not capture the global  minimum.  Consensus-based sampling (CBS) \cite{carrillo2022consensus_sampling} utilizes a single homogeneous swarm that can operate in two distinct modes. In its sampling regime, the swarm performs exploration and Gibbs amplification. However, it lacks an automatic mechanism to capture a good basin and provide feedback, and requires explicit parameter tuning to force the swarm into an optimizer mode.

Similarly, in ensemble Kalman methods and their Langevin variants, exploration and feedback are provided by the empirical covariance or cross-covariance of a swarm, often supplemented by Brownian noise or artificial inflation of covariance. While the covariance updates provide natural amplification of promising directions, they would achieve capture through ensemble collapse possibly induced by an explicit annealing process. The mass driven annealing approach in \cite{ding2024swarm} is another route which leads to the event of ``capture". By coupling the particles' positions with a dynamic mass-transfer mechanism, the swarm self-organizes into slow-moving exploiters and fast-moving explorers. However, unlike CBO, additive CBO, CBS, EKI (ensemble Kalman inversion) or TSO, the dynamics is gradient dependent.   

In comparison, TSO uses $N$ explorers and one treasure-hunter to portray the exploratory and exploitative behaviors, respectively, but it does not strictly isolate them. Rather, TSO establishes a continuous feedback loop: the swarm's exploration feeds the treasure hunter's jumps, and the hunter's exploitation anchors the swarm's future search, and thus they both navigate the landscape to find treasure.

\paragraph{Multi-species explorers.} The hunter-explorer architecture is not restricted to a single homogeneous swarm. One may instead partition the explorers into $S$ species, indexed by $s \in \{1,\dots,S\}$, each evolving under its own interaction drift $b^{(s)}$ while sharing the common hunter $Y_t$. An explorer of species $s$ may then evolve according to
\begin{align}\label{tso_eqn_multispecies}
    \der X_t^{i,s,N_s}
    &=  b^{(s)}
    \bigl(X_t^{i,s,N_s},Y_t,\mathcal{E}_t^{1,N_1},\dots,\mathcal E_t^{S,N_S}\bigr)\der t + \sigma^{(s)}\der W_t^{i,s} \nonumber\\ &\quad + \int_{\mathbb R^d}
    \Gamma^{(s)}
    \bigl(X_{t^-}^{i,s,N_s},Y_t,\mathcal{E}_t^{s,N_s},z\bigr)
    N^{i,s}(\der t,\der z).
\end{align}
 This heterogeneity allows different species to specialize, some favoring broad dispersal and others local refinement, while the single hunter continues to act as a macroscopic variable coupling back into every species. The dependence of $b^{(s)}$ on the vector
$(\mathcal E_t^{1,N_1},\dots,\mathcal E_t^{S,N_S})$ is meant to allow both intra-species and cross-species interactions. In particular, a given species may depend only on its own empirical measure, while more general choices may couple different species through repulsion, covariance information, or other cross-population interaction mechanisms. We restrict attention to the single-species case in this paper and regard the multi-species dynamics \eqref{tso_eqn_multispecies} as a promising direction for future study.

Subsection~\ref{tso_sec_mean_field} below is devoted to the derivation and well-posedness of the conditional McKean-Vlasov mean-field limit and Subsection~\ref{subsec2.2_conv_rep} provides the convolution representation of conditional law of explorers. 

 \subsection{Mean-field limit} \label{tso_sec_mean_field}
In TSO, the treasure hunter is a macroscopic variable shared by all explorers. Hence the  explorers, as $N \rightarrow \infty$, are not independent unconditionally. Rather, they are conditionally independent given the common randomness generated by the hunter's dynamics. We therefore formulate the mean-field law as a random conditional law.
Let $ \mathcal F_t^0:=\sigma(Y_0,N_s^Y:0\le s\leq t)$,
completed in the usual way, denote the common-noise filtration generated by the hunter's initial condition and Poisson clock. The mean-field law of a representative explorer is defined by
\begin{align}
    \mu_t:=\mathcal L(X_t\mid \mathcal F_t^0).
\end{align}
Equivalently, $\mu_t$ is the $\mathcal F_t^0$-measurable random probability measure characterized by
\begin{align}
    \int_{\mathbb R^d}\varphi(x)\,\mu_t(\der x)
    =
    \mathbb E[\varphi(X_t)\mid\mathcal F_t^0],
    \qquad
    \text{for all bounded measurable }\varphi:\mathbb R^d\to\mathbb R.
\end{align}
Consequently, the consensus point $\mathfrak m^\alpha(\mu_t)$ is itself a random, $\mathcal F^0_t$-adapted process.

 Let $N^{Y}(\der t, \der z)$ denote Poisson random measure on $[0,T]\times \{1\}$ with finite activity and intensity measure $\lambda_Y \delta_1 (\der z)\der t$. 
The mean-field representative explorer satisfies the conditional McKean-Vlasov jump-diffusion SDE:
\begin{align}
\der X_t
&= -\eta_1\bigl(X_t -\mathfrak{m}^\alpha(\mu_t)\bigr)\der t - \eta_2 (X_t- Y_t) \der t + 
\sigma\der W_t \nonumber \\
&\quad + \int_{\mathbb R^d}
\Big((1-\cos(\phi))
\big(
-X_{t^-}+ \mathrm{M}^\alpha(\mu_t,Y_t)
\big) + \sigma_J\sin(\phi)z\Big)
N(\der t,\der z), \label{tso_eqn_mf}
\end{align}
where $N(\der t,\der z)$ is a Poisson random measure with intensity $\lambda_J(2\pi)^{-d/2}e^{-|z|^2/2}\der z \der t$, $\mathrm{M}^\alpha(\mu_t, Y_t) = (\kappa_1 \m^\alpha(\mu_t) + \kappa_2 Y_t)/(\kappa_1 + \kappa_2) $ and the mean-field treasure hunter $Y$ is governed by
\begin{align}\label{tso_eqn_y_meanf}
    \der Y_t = -\beta\big(Y_t -\mathfrak{m}^\alpha(\mu_t)\big) \der t  + \int_{\{1\}}\bigl(\mathfrak{m}^\alpha(\mu_t)-Y_{t^-}\bigr)
    \mathbf 1_{\{f(\mathfrak{m}^\alpha(\mu_t))<f(Y_{t^-})\}} N^{Y}(\der t, \der z). 
\end{align}
The corresponding weak form for the explorer law is the following. For every
test function $\varphi\in C_b^2(\mathbb R^d)$,
\begin{align*}
\frac{\der }{\der t}
\int_{\mathbb R^d}\varphi(x)\,\mu_t(\der x)
= \int_{\mathbb{R}^d}
\mathcal A_{\mu_t}\varphi(x)\,\mu_t(\der x),
\end{align*}
where
\begin{align*}
\mathcal A_{\rho}\varphi(x)
& =
-\eta_1\bigl( (x-\mathfrak{m}^\alpha(\rho))\cdot\nabla\varphi(x)\bigr) -\eta_2\bigl( (x-Y_t)\cdot\nabla\varphi(x)\bigr)
+
\frac{\sigma^2}{2}\Delta\varphi(x)
\nonumber \\   &   \;\;\;+
\lambda_J
\int_{\mathbb R^d}
\left[\varphi(
\mathcal{J}_{\rho, Y_t}(x,z) ) - \varphi(x)
\right] \nu(\der z).
\end{align*}
Here  $\mathcal{J}_{\rho,y}(x,z) = \cos(\phi)x + ( 1- \cos(\phi) )\mathrm{M}^\alpha(\rho, y)+\sigma_J\sin(\phi)z $ with $\mathrm{M}^{\alpha}(\rho, y) = (\kappa_1 \mathfrak{m}^{\alpha}(\rho) + \kappa_2 y)/(\kappa_1 + \kappa_2) $.
In compact form, formally, we can write
\begin{align*}
\partial_t\mu_t =
\eta_1\big(\nabla\cdot (
(x-\mathfrak{m}^\alpha(\mu_t))\mu_t
)\big) + \eta_2\big(\nabla \cdot ((x- Y_t)\mu_t)\big) + 
\frac{\sigma^2}{2}\Delta\mu_t
+ \lambda_J
\left(\mathcal J_{\mu_t,Y_t\#}(\mu_t\otimes\nu)
- \mu_t \right),
\end{align*}
where $\mathcal{J}_{\rho,y\#}(\rho\otimes\nu)$ denotes the pushforward of
$\rho\otimes\nu$ by the map $\mathcal J_{\rho,y}$. Thus, the mean-field TSO is the coupled system
\begin{align}
\partial_t\mu_t
&=
\eta_1\nabla\cdot
\left(
\bigl(x-\mathfrak{m}^\alpha(\mu_t)\bigr)\mu_t
\right) + \eta_2\nabla\cdot
\left(
\bigl(x- Y_t\bigr)\mu_t
\right) + \frac{\sigma^2}{2}\Delta\mu_t
+\lambda_J
\left(\mathcal J_{\mu_t,Y_t\#}(\mu_t\otimes\nu) - \mu_t
\right), \\
\der Y_t  &= -\beta (Y_t - \mathfrak{m}^{\alpha}(\mu_t))\der t + \int_{\{1\}}( -Y_{t^-} + \mathfrak{m}^{\alpha}(\mu_t))  \mathbf 1_{\{f(\mathfrak{m}^\alpha(\mu_t))<f(Y_{t^-})\}} N^Y(\der t,  \der z).
\end{align}

\paragraph{Probabilistic framework.}
 Let $ (\Omega^0,\mathcal{F}^0,\mathbb P^0,\mathbb{F}^0),\, \mathbb F^0=(\mathcal{F}_t^0)_{t\in[0,T]}$,
be a complete filtered probability space satisfying the usual conditions and supporting the common variables $
    Y_0$ and $N^Y=(N_t^Y)_{t\in[0,T]}$,
where $N^Y$ is a Poisson process with intensity $\lambda_Y>0$. Let $
    (\Omega^1,\mathcal F^1,\mathbb{P}^1,\mathbb F^1),\,
    \mathbb{F}^1=(\mathcal F_t^1)_{t\in[0,T]}$,
be another complete filtered probability space satisfying the usual conditions and supporting the idiosyncratic random elements $
    X_0,\, W,\, N(\der t,\der z)$ in \eqref{tso_eqn_mf}. We work on the completed product space
\begin{align}
    (\Omega,\mathcal F,\mathbb P)
    :=
    (\Omega^0\times\Omega^1,
    \mathcal{F}^0\otimes\mathcal F^1,
    \mathbb P^0\otimes\mathbb{P}^1),
\end{align}
with the usual augmentation of $\mathcal F_t:=\mathcal F_t^0\otimes\mathcal F_t^1$.
The common-noise filtration is identified with $\mathbb F^0$ on the product space. For a random variable
$Z:\Omega^0\times\Omega^1\to\mathbb R^d$, we write
\begin{equation}
    \mathcal L^1(Z)(\omega^0)
    :=
    \mathbb P^1\circ Z(\omega^0,\cdot)^{-1}.
\end{equation}
Thus, for every bounded measurable $\varphi:\mathbb R^d\to\mathbb R$,
\begin{align}
    \int_{\mathbb R^d}\varphi(x)\,\mathcal L^1(Z)(\omega^0)(\der x)
    =    \int_{\Omega^1}\varphi(Z(\omega^0,\omega^1))\,\mathbb P^1(\der \omega^1).
\end{align}
If $Z_t$ is $\mathcal F_t$-measurable, then $\mathcal L^1(Z_t)$ is an
$\mathcal F_t^0$-measurable $\mathcal P(\mathbb R^d)$-valued random variable and is a version of
$\mathcal L(Z_t\mid\mathcal F_t^0)$. This product-space representation of conditional laws in the common-noise framework  is provided in \cite[Appendix~A]{hernandez2025conditional} and in \cite{carmona2018probabilistic} for c\`{a}dl\`{a}g processes and continuous processes, respectively. Moreover, if for $\mathbb P^0$-a.e. $\omega^0$
\begin{align}
    \lim_{s\to t}\mathbb E^1
    \bigl[
        |Z_t(\omega^0,\cdot)-Z_s(\omega^0,\cdot)|^2
    \bigr]=0,
\end{align}
then
\begin{equation}
    t\mapsto \mathcal{L}^1(Z_t)(\omega^0)
\end{equation}
is continuous in $\mathcal P_2(\mathbb R^d)$, since
\begin{align}
    \mathcal W_2^2\bigl(\mathcal L^1(Z_t)(\omega^0),\mathcal L^1(Z_s)(\omega^0)\bigr)
    \le
    \mathbb E^1
    \bigl[
        |Z_t(\omega^0,\cdot)-Z_s(\omega^0,\cdot)|^2
    \bigr].
\end{align}
Here and below, $\mathbb E^1$ denotes expectation with respect to $\mathbb P^1$ only.

Within this probabilistic framework, we state the well-posedness result for the smoothed mean-field SDEs, where the indicator function is replaced by a smooth approximation.

\begin{theorem}
  Assume that $f:\mathbb R^d\to\mathbb R$ is bounded from below and satisfies
\begin{align}\label{tso_eqn_f_lowerbnd_2.14}
|f(x)-f(y)|\le L_f(|x|+|y|)|x-y|,\qquad
f(x)- f_{\min}\le c_u(1+|x|^2),
\end{align}
where $f_{\min} := \inf_{\mathbb R^d}f$. Assume that there are constants $M,c_l>0$ such that 
\begin{align}\label{tso_eqn_f_lowerbnd_2.15}
f(x)- f_{\min} \geq c_l |x|^2,\qquad |x|\ge M.
\end{align}
Let $\Psi_\varepsilon:\mathbb R^d\times\mathbb R^d\to[0,1]$ be a smooth approximation of
$
(m,y)\longmapsto \mathbf 1_{\{f(m)<f(y)\}},
$
and assume that the smoothed jump size  $G_\varepsilon(m,y):=(m-y)\Psi_\varepsilon(m,y)
$
is locally Lipschitz and has linear growth:
\begin{align}
|G_\varepsilon(m,y)-G_\varepsilon(\bar m,\bar y)| \le L_\varepsilon(|m-\bar m|+|y-\bar y|),\qquad
|G_\varepsilon(m,y)|\le C_\varepsilon(1+|m|+|y|).
\end{align}
Assume that $W$, $N$, $X_0$ are mutually independent and also independent of $Y_0, \; N^Y$, and that $ \mathbb E\bigl[|X_0|^4+|Y_0|^4\bigr]<\infty$. 
Let $\mathcal{F}^0_t:=\sigma(Y_0,N_s^Y:0\leq s\le t)$
completed in the usual sense. Then there exists a unique c\`{a}dl\`{a}g adapted solution
$(X_t,Y_t)_{t\in[0,T]}$ of the smoothed mean-field system corresponding to \eqref{tso_eqn_mf}-\eqref{tso_eqn_y_meanf} satisfying
\begin{align}
    \sup_{t\in[0,T]}
    \mathbb E\bigl[|X_t|^4+|Y_t|^4\bigr]<\infty.
\end{align}
Moreover, $t\mapsto\mu_t (\omega^0)$ belongs to $C([0,T];\mathcal P_2(\mathbb R^d))$ for $\mathbb P^0$-a.e. $\omega^0$.
\end{theorem}

 We follow the Leray-Schauder argument as used in \cite{carrillo2018analytical} for the standard mean-field CBO equation, with some important modifications. The mean-field law is now the conditional law with respect to the common-noise filtration. The relevant fixed point is therefore constructed pathwise with respect to the common Poisson clock. After proving existence and uniqueness of a fixed point for each common path, we apply Kuratowski-Ryll-Nardzewski selection theorem (see \cite{ch1977convex}) to construct an $\mathbb{F}^0$-adapted stochastic process. 

 Measurable-selection arguments of this type are used in the analysis for conditional McKean-Vlasov SDEs, see for instance \cite{lacker2023superposition} for measurable selection in conditional McKean-Vlasov equations and  \cite{santambrogio2021cucker} for an explicit use in a mean-field game setting. We emphasize, however, that these works do not cover the present setting, where the fixed point is first constructed pathwise with respect to a finite-activity common Poisson clock.

Before beginning the proof, we recall the required stability estimate for the weighted average. Under
\eqref{tso_eqn_f_lowerbnd_2.14}-\eqref{tso_eqn_f_lowerbnd_2.15}, for every $K>0$ there exists a constant
$c_K>0$, depending only on $\alpha$, $L_f$, $c_u$, and $K$, such that, whenever
$\rho,\bar\rho\in\mathcal{P}_4(\R^d)$ satisfy \cite[Lemma~3.2]{carrillo2018analytical}
\begin{align}
    \int_{\mathbb R^d}|x|^4\,\rho(\der x)
    +
    \int_{\mathbb{R}^d}|x|^4\,\bar\rho(\der x)
    \le K,
\end{align}
one has
\begin{equation}
    |\m^\alpha(\rho)-\m^\alpha(\bar\rho)|
    \leq
    c_K \mathcal W_2(\rho,\bar\rho).
    \label{eq:local_lip_m}
\end{equation}
Further, under the condition \eqref{tso_eqn_f_lowerbnd_2.15}, there exist constants $b_1,b_2>0$ such that \cite[Lemma~3.3]{carrillo2018analytical}
\begin{equation}
    |\m^\alpha(\rho)|^2
    \leq
    \int_{\mathbb{R}^d}|x|^2
    \frac{w^{\alpha,f}(x)\rho(\der x)}
    {\int_{\mathbb R^d}w^{\alpha,f} (x)\rho(\der x)}
    \le
    b_1+b_2\int_{\mathbb R^d}|x|^2\,\rho(\der x).
    \label{tso_eqn_m_growth_quadratic}
\end{equation}

\begin{proof}
Let $\omega^0$ be a fixed realization of the common randomness, i.e. of
$(Y_0,N^Y)$, outside a null set. On $[0,T]$, the corresponding Poisson path has only finitely many jumps. We construct the solution conditionally on this common path. For notational simplicity, the dependence on $\omega^0$ will be suppressed in several estimates.

Let $u\in C([0,T];\mathbb R^d)$ be fixed.
 We first consider the following system:
\begin{align}
\der Y_t^u &=
-\beta(Y_t^u-u_t)\der t
+ G_\varepsilon(u_{t-},Y_{t-}^u)\der N_t^Y,
\label{tso_eqn_frozen_Y}
\\
\der X_t^u
&=
-\eta_1(X_t^u-u_t)\der t
-\eta_2(X_t^u-Y_t^u)\der t
+\sigma\der W_t
+\int_{\mathbb{R}^d}
\Gamma(u_t,X_{t-}^u,Y_t^u,z)N(\der t,\der z),
\label{tso_eqn_frozen_X}
\end{align}
where $\Gamma(u,x,y,z)
    :=  (1-\cos\phi)
    \left(
    -x+\frac{\kappa_1u+\kappa_2y}{\kappa_1+\kappa_2}
    \right)
    +    \sigma_J\sin\phi\,z$.
For fixed $u$ and fixed common Poisson path in treasure hunter's dynamics, the coefficients in \eqref{tso_eqn_frozen_X} and drift coefficient in \eqref{tso_eqn_frozen_Y} 
are Lipschitz in $(x,y)$, and jump coefficient in \eqref{tso_eqn_frozen_Y} is locally Lipschitz. All coefficients have linear growth. Since the jump intensities are finite and the Gaussian jump-size distribution has finite moments of all orders, the well-posedness theorem \cite[Theorem~1]{gyongy1980stochastic} for jump-diffusion SDEs gives a unique strong c\`{a}dl\`{a}g solution
$(X^u,Y^u)$. Moreover, we have
\begin{align}
    \E^1|X^u_t|^{4} < \infty. 
\end{align}
Let
\begin{align}
    \nu_t^u:=\mathcal{L}^1(X_t^u)
\end{align}
be the law of the frozen explorer $X^u_t$. In other words, after fixing the common path $\omega^0$, $\nu_t^u$  represents the law of $X_t^u$ with respect to the idiosyncratic noises $W$ (Wiener noise) and $N$ (Poisson noise and Gaussian jumps)  defined on $(\Omega^1, \mathbb{F}^1, \mathbb{P}^1)$. Define
\begin{equation*}
    (\mathcal{T} u)_t
    :=
    \m^\alpha(\nu_t^u),
    \qquad 0\le t\le T.
\end{equation*}
We first show that $\mathcal T:C([0,T];\mathbb R^d)\to C([0,T];\mathbb R^d)$ is continuous and compact.

Let $\|u\|_\infty: = \sup_{t \in [0,T]} |u_t| \leq R$ for $u \in C([0,T];\R^d)$. For each $\omega^0$, we treat $C_R(\omega^0)$ as generic constant changing from line to line. By the linear growth of $G_\varepsilon$ and finite number of jumps in path of $Y^u$, Gronwall's inequality gives
\begin{equation}
    \sup_{t\in[0,T]}|Y_t^u|
    \le C_R(\omega^0).
    \label{tso_eq:Y_bound_frozen}
\end{equation}
Using Ito's formula for jump processes to $|X_t^u|^4$, using the linear growth of $\Gamma$, the bound \eqref{tso_eq:Y_bound_frozen}, and the finiteness of the moments of the Gaussian mark distribution, we get
\begin{equation}
    \sup_{t\in[0,T]}
    \E^1|X_t^u|^4
    \le
    C_R(\omega^0).
    \label{eq:X_fourth_bound_frozen}
\end{equation}
Therefore, by \eqref{tso_eqn_m_growth_quadratic}, we have
\begin{equation*}
    \| \mathcal T u\|_\infty
    \le C_R(\omega^0).
\end{equation*}
Using Ito's formula and the moment bounds obtained above for $X^u$, for $0\le s<t\le T$, we get
\begin{equation}
    \mathbb E^1|X_t^u-X_s^u|^2
    \leq  C_R(\omega^0)|t-s|.
    \label{eq:X_increment_frozen}
\end{equation}
Indeed, the finite-variation drift contributes $|t-s|^2$, the Brownian increment contributes order $|t-s|$, and the finite-activity jump integral contributes order $|t-s|$. The above computation is not difficult and has been done in \cite{carrillo2018analytical} and \cite{kalise2023consensus} for diffusion only and jump-diffusion SDEs, respectively. Therefore
\begin{align*}
    \mathcal{W}_2(\nu_t^u,\nu_s^u)
    \le
    \bigl(\mathbb E^1|X_t^u-X_s^u|^2\bigr)^{1/2}
    \le
    C_R(\omega^0)|t-s|^{1/2}.
\end{align*}
Using \eqref{eq:local_lip_m} with the fourth-moment bound \eqref{eq:X_fourth_bound_frozen}, we obtain
\begin{align*}
    |(\mathcal T u)_t-(\mathcal T u)_s|
    \le
    C_R(\omega^0)|t-s|^{1/2}.
\end{align*}
Thus, $\mathcal T$ maps bounded subsets of $C([0,T];\mathbb R^d)$ into bounded subsets of
$C^{0,1/2}([0,T];\mathbb R^d)$. By Arzel\`{a}-Ascoli theorem, $\mathcal T$ is compact.

We now prove continuity.  Let $u^n\to u$ in
$C([0,T];\mathbb R^d)$. In what follows, we will denote by $C(\omega^0)$ a generic constant for a given realization $\omega^0$. Denote fixed Poisson path's jump times on $[0,T]$ by $0<\tau_1<\cdots<\tau_J\le T$. Since $u^n\to u$ uniformly, the
sequence $(u^n)_n$ is uniformly bounded. The linear-growth of $G_\varepsilon$, together with Gronwall's inequality and the finiteness of the number of jumps, gives
\begin{equation} \label{tso_eqn_Y2.31}
    \sup_n\sup_{t\in[0,T]}|Y_t^{u^n}|
    +     \sup_{t\in[0,T]}|Y_t^u|
    \leq C(\omega^0).
\end{equation}
 Writing $\Delta_t^n:=Y_t^{u^n}-Y_t^u$ then between two consecutive jump times, we have
\begin{align*}
    |\Delta_t^n|
    \leq |\Delta_{\tau_k}^n| + \beta\int_{\tau_k}^t|\Delta_s^n|\der s + C\|u^n-u\|_\infty, \quad \tau_k \leq t < \tau_{k+1}.
\end{align*}
At a jump time $\tau_k$, $|\Delta_{\tau_k}^n|
    \leq C\bigl(
        |\Delta_{\tau_k-}^n|
        + \|u^n-u\|_\infty \bigr)$ where we have used Lipschitzness of $G_\varepsilon$ and \eqref{tso_eqn_Y2.31}. Iterating over the finitely many jump intervals gives
\begin{equation*}
    \sup_{t\in[0,T]}|Y_t^{u^n}-Y_t^u|
    \leq     C(\omega^0)\|u^n-u\|_\infty,
\end{equation*}
and hence
\begin{equation*}
    \sup_{t\in[0,T]}|Y_t^{u^n}-Y_t^u|\to0.
\end{equation*}
In the similar manner, using Lipschitzness of the coefficients of  $X^{u^n}$ and $X^u$, we apply standard procedure via Ito's formula for jump-diffusions to get
\begin{equation*}
    \sup_{t\in[0,T]}
    \mathbb E^1|X_t^{u^n}-X_t^u|^2
    \to0.
\end{equation*}
Consequently,
\begin{equation} \label{tso_eqn_2.35}
    \sup_{t\in[0,T]}
    \mathcal W_2(\nu_t^{u^n},\nu_t^u)\rightarrow 0.
\end{equation}
 Therefore \eqref{tso_eqn_2.35} and \eqref{eq:local_lip_m} imply
\begin{equation*}
    \|\mathcal{T} u^n-\mathcal T u\|_\infty\to0.
\end{equation*}
Thus, $\mathcal T$ is continuous.
Let $u\in C([0,T];\mathbb R^d)$ and $\theta\in[0,1]$ satisfy $ u=\theta\mathcal{T} u$. Then, by \eqref{tso_eqn_m_growth_quadratic},
\begin{equation}
    |u_t|^2 = \theta^2 | \mathcal{T}(u)_t|^2 = \theta^2 |\m^\alpha(\nu_t^u)|^2  
    \le
    C(1+\mathbb E^1|X_t^u|^2).
    \label{eq:LS_u_bound}
\end{equation}
Let $0<\tau_1<\cdots<\tau_J\leq T $
be the jump times of the fixed common Poisson path $N^Y(\omega^0)$ on
$[0,T]$, and set $\tau_0:=0$, $\tau_{J+1}:=T$. Define $S_t:=\E^1|X_t^u|^2+|Y_t^u|^2$.  Fix $j\in\{0,\ldots,J\}$ and let $t\in[\tau_j,\tau_{j+1})$. On this interval there is no jump of $N^Y$, and hence, by Young's inequality, we have
\begin{align}
    |Y_t^u|^2
    &\le 2|Y_{\tau_j}^u|^2 + 2\beta^2
    \Big| \int_{\tau_j}^t (Y_s^u-u_s)\der s
    \Big|^2
    \nonumber\\
    &\leq 2|Y_{\tau_j}^u|^2
    +  2\beta^2(t-\tau_j)
    \int_{\tau_j}^t |Y_s^u-u_s|^2\der s
\le C\Big(|Y_{\tau_j}^u|^2 + \int_{\tau_j}^t
        \bigl(|Y_s^u|^2+|u_s|^2\bigr)\der s\Big).
    \label{eq:Y_between_jumps_estimate}
\end{align}
Using Ito's isometry for the
Brownian integral term and the second-moment estimate for finite-activity
Poisson integrals, we obtain
\begin{align}
    \E^1|X_t^u|^2
    &\le C\E^1|X_{\tau_j}^u|^2
    + C\mathbb E^1
    \Big|\int_{\tau_j}^t
        \big(\eta_1(X_s^u-u_s) + \eta_2(X_s^u-Y_s^u)\big)\der s \Big|^2 \nonumber\\
    &\quad   + C\sigma^2(t-\tau_j) + C\mathbb E^1
    \Big|\int_{\tau_j}^t\int_{\mathbb R^d}
        \Gamma(u_s,X_{s-}^u,Y_{s-}^u,z)\,N(\der s,\der z)\Big|^2
    \nonumber\\ 
    &\le C\Big(1+\E^1|X_{\tau_j}^u|^2
        + \int_{\tau_j}^t
        \big(\E^1|X_s^u|^2 + |Y_s^u|^2 +  |u_s|^2
        \big)\der s
    \Big).
    \label{eq:X_between_jumps_estimate}
\end{align}
In the last inequality, we have used the linear growth bound $|\Gamma(u_s,x,y,z)|^2 \leq C(1+|x|^2+|y|^2+|u_s|^2+|z|^2) $ 
and the facts that $\int_{\mathbb R^d}|z|^2\,\nu_J(\der z)<\infty$ and $\nu_J(\mathbb R^d)<\infty$. Combining \eqref{eq:Y_between_jumps_estimate} and
\eqref{eq:X_between_jumps_estimate}, and then using
\eqref{eq:LS_u_bound}, gives
\begin{align}
    S_t &\leq  C\Big(
        1+S_{\tau_j} + \int_{\tau_j}^t (S_s+|u_s|^2)\der s
    \Big) \leq
    C\Big(
        1+S_{\tau_j}
        +
        \int_{\tau_j}^t S_s\der s
    \Big).
    \label{eq:S_between_jumps}
\end{align}
By Gronwall's inequality on the interval $[\tau_j,\tau_{j+1})$, we obtain
\begin{equation}
    \sup_{t\in[\tau_j,\tau_{j+1})}S_t
    \le
    C\bigl(1+S_{\tau_j}\bigr).
    \label{eq:S_interval_bound}
\end{equation}
We now include the effect of a hunter's common jump to get the final bound. At a jump time $\tau_j$ of $N^Y$, we have
\begin{align*}
    Y_{\tau_j}^u = Y_{\tau_j-}^u
    + G_\varepsilon(u_{\tau_j},Y_{\tau_j-}^u).
\end{align*}
Since $G_\varepsilon$ has linear growth, i.e. $|G_\varepsilon(m,y)|\le C(1+|m|+|y|)$, 
we get
\begin{align}
    |Y_{\tau_j}^u|^2 \le
    2|Y_{\tau_j-}^u|^2
    +
    2|G_\varepsilon(u_{\tau_j},Y_{\tau_j-}^u)|^2
    \leq
    C\bigl(
        1+|Y_{\tau_j-}^u|^2+|u_{\tau_j}|^2
    \bigr).
\end{align}
Using \eqref{eq:LS_u_bound}, the above inequality implies
\begin{align}
    S_{\tau_j}
    &=
    \mathbb E^1|X_{\tau_j}^u|^2+|Y_{\tau_j}^u|^2
    \leq
    C\big(
        1+\mathbb E^1|X_{\tau_j}^u|^2+|Y_{\tau_j-}^u|^2
    \big) \leq
    C\bigl(1+S_{\tau_j-}\bigr).
    \label{eq:S_jump_bound}
\end{align}

Starting from $S_0=\mathbb E^1|X_0|^2+|Y_0|^2<\infty$, 
we first apply \eqref{eq:S_interval_bound} on $[0,\tau_1)$, then
\eqref{eq:S_jump_bound} at $\tau_1$, then again
\eqref{eq:S_interval_bound} on $[\tau_1,\tau_2)$ and proceed iteratively. Since the number
$J:=J_T^Y(\omega^0)$ of common jumps on $[0,T]$ is finite, this induction gives
\begin{equation*}
    \sup_{t\in[0,T]}S_t
    \le
    C(\omega^0)\bigl(1+S_0\bigr)
    <\infty .
\end{equation*}
The constant $C(\omega^0)$ may depend on $T$, the model parameters, and the number
and locations of the common jumps, but it is independent of $u$. Finally, by \eqref{eq:LS_u_bound}, we have
\begin{align*}
    \|u\|_\infty^2
    =
    \sup_{t\in[0,T]}|u_t|^2
    \leq
    C\big(1+\sup_{t\in[0,T]}S_t\big)
    \leq
    C(\omega^0).
\end{align*}
This proves the boundedness of the Leray-Schauder set
\begin{equation}
    \left\{
        u\in C([0,T];\mathbb R^d):
        u=\theta\mathcal T u
        \text{ for some }\theta\in[0,1]
    \right\}.
\end{equation}
Since $\mathcal T$ is continuous and compact, the Leray-Schauder fixed-point theorem yields a fixed point
$m\in C([0,T];\mathbb R^d)$ such that
\begin{equation}
    m_t=(\mathcal Tm)_t
    =
    \mathfrak m^\alpha(\mathcal L^1(X_t^m)).
\end{equation}
Setting $Y:=Y^m$ and $X:=X^m$, we have shown the existence of at least one solution for every fixed common-noise realization, satisfying
\begin{align}
    m_t =     \m^\alpha(\mathcal L^1(X_t)).
\end{align}
 Let $(X,Y)$ and $(\bar{X},\bar Y)$ be two solutions driven by the same idiosyncratic noises and the same common hunter's Poisson path, with the same initial conditions. Let $\mu_t:=\mathcal L^1(X_t)$, $\bar{\mu}_t:=\mathcal L^1(\bar X_t)$, $m_t:=\m^\alpha(\mu_t)$ and $\bar m_t:=\m^\alpha(\bar\mu_t)$.
Using \eqref{eq:X_fourth_bound_frozen}, for the fixed common path there exists $K(\omega^0)<\infty$ such that
\begin{align}
    \sup_{t\in[0,T]}
    \left(\int_{\mathbb R^d}|x|^4\mu_t(dx) + \int_{\mathbb R^d}|x|^4\bar\mu_t(dx)
    \right)
    \le K(\omega^0).
\end{align}
The local stability estimate \eqref{eq:local_lip_m} therefore gives
\begin{align*}
    |m_t-\bar m_t|^2
    \leq
    C(\omega^0)
    \mathcal W_2^2(\mu_t,\bar\mu_t).
\end{align*}
Since the conditional law of $(X_t,\bar X_t)$ over $\Omega^1$ is a coupling of $\mu_t$ and $\bar\mu_t$,
\begin{equation}
    |m_t-\bar m_t|^2
    \le
   C(\omega^0)
    \mathbb E^1|X_t-\bar X_t|^2.
    \label{eq:m_difference_conditional}
\end{equation}
Let $\Delta X_t:=X_t-\bar X_t$, $\Delta Y_t:=Y_t-\bar Y_t$, and $\Delta m_t:=m_t-\bar m_t$.
Applying Ito's formula to $|\Delta X_t|^2$ and using Ito's isometry, we get
\begin{align}
    \mathbb E^1|\Delta X_t|^2
    \leq
    C(\omega^0)
    \int_0^t
    \left(
        \mathbb E^1|\Delta X_s|^2
        +
        |\Delta Y_s|^2
        +
        |\Delta m_s|^2
    \right)\der s.
\end{align}
Between common jump times,
\begin{align*}
    \frac{\der }{\der t}|\Delta Y_t|^2
    \leq
    C(\omega^0)
    \bigl(
        |\Delta Y_t|^2+|\Delta m_t|^2
    \bigr).
\end{align*}
At a common jump time $\tau$, $\Delta Y_\tau
    =\Delta Y_{\tau-} + G_\varepsilon(m_{\tau},Y_{\tau-})
    - G_\varepsilon(\bar m_{\tau},\bar Y_{\tau-})$,
and the local Lipschitz property of $G_\varepsilon$ gives
\begin{align*}
    |\Delta Y_\tau|^2
    \le
    C(\omega^0)
    \bigl(
        |\Delta Y_{\tau-}|^2+|\Delta m_{\tau}|^2
    \bigr).
\end{align*}
Combining these estimates with \eqref{eq:m_difference_conditional}, and again using the fact that the common Poisson path has finitely many jumps, yields
\begin{align}
    \mathbb E^1|\Delta X_t|^2+|\Delta Y_t|^2
    \leq
    C(\omega^0)
    \int_0^t
    \left(\mathbb E^1|\Delta X_s|^2+|\Delta Y_s|^2
    \right)\der s.
\end{align}
Applying Gronwall's inequality gives
\begin{align}
    \mathbb E^1|\Delta X_t|^2+|\Delta Y_t|^2=0,
    \qquad
    t\in[0,T].
\end{align}
Hence $X=\bar X$ and $Y=\bar Y$, for the fixed common path $\omega^0$. Consequently, the fixed point $m$ is unique for almost every common path in $\Omega^0$.

It remains to demonstrate that these pathwise fixed points generate an adapted stochastic process. 
Let
\begin{equation*}
    E^0:=\R^d\times D([0,T];\mathbb N_0)
\end{equation*}
be the space of common paths, and let
\begin{align*}
    \zeta(\omega^0):=(Y_0(\omega^0),N^Y(\omega^0)).
\end{align*}
For each $e\in E^0$ outside a $\mathbb P^0\circ\zeta^{-1}$-null set, let
\begin{align*}
    \mathcal T^e:C([0,T];\R^d)\to C([0,T];\R^d)
\end{align*}
be the pathwise fixed-point map constructed above, and define
\begin{align*}
    \mathcal A(e):=
    \{u\in C([0,T];\mathbb R^d):u=\mathcal T^e u\}.
\end{align*}
The Leray-Schauder theorem based argument shows that $\mathcal A(e)\neq\emptyset$. Since
$\mathcal T^e$ is continuous in $u$, the set $\mathcal A(e)$ is closed in
$C([0,T];\mathbb R^d)$. Moreover, the map
\begin{align*}
    (e,u)\mapsto \mathcal T^e u
\end{align*}
is Borel measurable (see Lemma~\ref{tso_borel_measurability_lemma}). Hence
\begin{align*}
    \operatorname{Graph}(\mathcal A)= \{(e,u) \in E^0 \times C([0,T]; \R^d) :u=\mathcal T^e u\}
    =
    \left\{(e,u):\|u-\mathcal T^e u\|_\infty=0\right\}
\end{align*}
is a measurable subset of
$E^0\times C([0,T];\mathbb R^d)$. By the
Kuratowski-Ryll-Nardzewski measurable selection theorem (see \cite[Section~5.2]{srivastava1998course}), there exists a
measurable map
\begin{align*}
    S:E^0\to C([0,T];\mathbb R^d)
\end{align*}
such that $S(e)\in\mathcal A(e)$ for every admissible common path $e$. We define
\begin{align*}
    m(\omega^0):=S(\zeta(\omega^0)).
\end{align*}
Then $m$ is a measurable $C([0,T];\mathbb R^d)$-valued random variable. 

It remains to show that $m$ is adapted. Let $e,\tilde e\in E^0$ coincide on
$[0,t]$. The fixed-point problems
corresponding to $e$ and $\tilde e$ coincide on $[0,t]$. Since pathwise
uniqueness holds on every subinterval, we have
\begin{align}
    S(e)(s)=S(\tilde e)(s),
    \qquad 0\le s\le t.
\end{align}
In particular, $S(e)(t)$ depends only on the stopped common path
$e_{\cdot\wedge t}$. Therefore $m_t$ is a measurable function of
$(Y_0,N^Y_{\cdot\wedge t})$, and hence $m_t$ is $\mathcal F_t^0$-measurable.
Thus $m$ is $\mathbb F^0$-adapted. This completes the proof.

\end{proof}

\begin{remark}
The well-posedness result is stated and proved for the specific consensus-hunter drift \eqref{tso_eqn_2.2} used throughout this paper. The argument, however, is not tied to this particular choice. Indeed, the proof relies only on constructing a pathwise fixed point conditional on the hunter's Poisson path and then getting an $\mathbb F^0$-adapted solution via a measurable selection. These steps require only suitable regularity of the
interaction drift and do not rely on the explicit consensus-hunter form.
\end{remark}

\begin{lemma}
\label{tso_borel_measurability_lemma}
Let $ E^0:=\mathbb R^d\times D([0,T];\mathbb N_0)$ endowed with its product Borel $\sigma$-field, where
$D([0,T];\mathbb{N}_0)$ carries the Skorokhod topology. For $e=(y_0,n)\in E^0$, let $\mathcal{T}^e$ denote the pathwise operator defined by
\begin{align*}
    (\mathcal{T}^e u)_t :=  \m^\alpha\bigl(\mathcal{L}^1(X_t^{u,e})\bigr),\qquad 0\le t\le T,
\end{align*}
where $(X^{u,e},Y^{u,e})$ is the solution of \eqref{tso_eqn_frozen_Y}-\eqref{tso_eqn_frozen_X} obtained by
fixing the common path $e=(y_0,n)$ and 
$u\in C([0,T];\mathbb R^d)$. Then the map
\begin{equation*}
    (e,u)\longmapsto \mathcal T^e u
\end{equation*}
is Borel measurable from $E^0\times C([0,T];\mathbb R^d)$  into $C([0,T];\mathbb R^d)$. 
\end{lemma}
\begin{proof}
    Thanks to the strong existence and pathwise uniqueness of the solution $(X^{u,e},Y^{u,e})$ of SDEs \eqref{tso_eqn_frozen_Y}-\eqref{tso_eqn_frozen_X}, there exists a measurable map $\Phi$ such that
    \begin{align}
        (X^{u,e}, Y^{u,e}) = \Phi(e, u, X_0, W, N_T)
    \end{align}
     where $N_T$ is a random element of the space of finite counting measures on $[0,T] \times \R^d$. Therefore, for every fixed $t$
\begin{equation*}
    (e, u, \omega^1) \longmapsto X_{t}^{u,e}(\omega^1)
\end{equation*}
     is measurable. Let $\varphi \in C_b(\R^d)$. Then
\begin{equation}
    \int_{\R^d}\varphi(x)\,
    \mathcal{L}^1(X_t^{u,e})(\der x)
    =     \int_{\Omega^1}
    \varphi\bigl(X_t^{u,e}(\omega^1)\bigr)
    \,\mathbb{P}^1(\der \omega^1).
\end{equation}
Since $(e,u,\omega^1)
    \longmapsto
    \varphi(X_t^{u,e}(\omega^1))$
is bounded and measurable, Fubini's theorem implies that
\begin{equation*}
    (e,u)
    \longmapsto     \int_{\Omega^1}
    \varphi\bigl(X_t^{u,e}(\omega^1)\bigr)
    \,\mathbb P^1(\der \omega^1)
\end{equation*}
is Borel measurable. Since the Borel $\sigma$-field on $\mathcal P(\mathbb R^d)$, endowed with the topology of weak convergence, is
 generated by the maps $ \rho\longmapsto \int_{\mathbb R^d}\varphi(x)\,\rho(\der x)$, $\varphi\in C_b(\mathbb R^d)$ (see \cite[Chapter~6]{parthasarathy2005probability}), 
it follows that
\[
    (e,u)\longmapsto \mathcal{L}^1(X_t^{u,e})
\]
is Borel measurable as a $\mathcal P(\mathbb R^d)$-valued map. Consequently, for every $t\in[0,T]$,
\begin{equation*}
    (e,u)
    \longmapsto  (\mathcal T^e u)_t
    =  \m^\alpha\bigl(\mathcal{L}^1(X_t^{u,e})\bigr)
\end{equation*}
is Borel measurable. In particular, this holds for every rational
$q\in\mathbb{Q}\cap[0,T]$. Since the paths $t\mapsto \mathcal{T}^e u$ are
continuous and since the Borel $\sigma$-field on
$C([0,T];\mathbb R^d)$ is generated by rational-time evaluations, we conclude
that $(e,u)\longmapsto \mathcal T^e u$
is Borel measurable from $E^0\times C([0,T];\mathbb R^d)$ into $C([0,T];\mathbb R^d)$. 
This completes the proof.
     
\end{proof}

\subsection{Convolution representation}\label{subsec2.2_conv_rep}

Let us assume that $X_0$ is independent of all driving noises and of the hunter filtration, and that
\begin{align}
\mathcal L(X_0\mid \mathcal F_0^0) &= \mathcal N(x_0,\Sigma_0), \quad \text{with} \qquad \Sigma_0\geq\sigma_0^2I_d
\end{align}
for some $\sigma_0>0$. Let $(\tau_i)_{i\geq1}$ be the jump epochs of the explorer's Poisson random measure $N(\der t,\der z)$, and let $(Z_i)_{i\geq1}$ be the associated jump marks. Since the jump-size distribution is standard Gaussian, the random variables $Z_n$ are i.i.d. with law $\mathcal N(0,I_d)$. They are independent of $X_0$, $W$, the hunter noise, and the unmarked explorer counting Poisson process. For $0\leq s\leq t$, let $J^X(s,t)$ denote the number of explorer jumps in the interval $(s,t]$. For brevity, define
\begin{align}
\Phi(s,t) &:= e^{-a(t-s)}c^{J^X(s,t)},
\end{align}
with $a := \eta_1 + \eta_2$ and $c := \cos(\phi)$. Applying the variation-of-constants formula between consecutive jump times gives, for every $t\geq0$,
\begin{align}
\begin{aligned}
X_t &= \Phi(0,t)X_0 + \int_0^t \Phi(s,t) [ \eta_1 \m^\alpha(\mu_s)+\eta_2Y_s]\der s  + \sum_{\tau_i\leq t} \Phi(\tau_i,t)(1-c) \mathrm{M}^\alpha(\mu_{\tau_i},Y_{\tau_i}) \\
&\quad + \sigma\int_0^t\Phi(s,t)\der W_s + \sum_{\tau_i\leq t}\Phi(\tau_i,t)\sigma_J \sin(\phi) Z_i .
\end{aligned}
\end{align}
Conditionally on the unmarked explorer jump process up to time $t$, the kernel $s \mapsto \Phi(s,t)$ is deterministic, and the stochastic integral above is understood as the corresponding Wiener integral. Now define the enlarged filtration
\begin{align}
\mathcal G_t &:= \mathcal{F}^0_t \vee \sigma\bigl(J^X(0,s):0\leq s\leq t\bigr),
\end{align}
where $J^X(0,s):=N((0,s]\times\mathbb R^d)$ is the unmarked explorer counting process. Thus, $\mathcal{G}_t$ contains the hunter's information and the explorer's jump times up to time $t$, but it does not contain the Brownian path or the Gaussian jump marks. Conditioned on $\mathcal G_t$, $\m^\alpha(\mu_s)$, $Y_s$, $\Phi(s,t)$, and the jump times are fixed, and therefore, $X_t$ is an affine function of the independent Gaussian random variables $X_0$, the Wiener integral term and $(Z_i)_{\tau_i\leq t}$. Therefore
\begin{align}
\mathcal L(X_t\mid\mathcal G_t) &= \mathcal N(M_t,V_t)
\end{align}
for some $\mathcal G_t$-measurable mean $M_t$. Its conditional covariance matrix is
\begin{align}
V_t &= \Phi(0,t)^2\Sigma_0 + \sigma^2\int_0^t\Phi(s,t)^2\der s\,I_d + \sigma_J^2\sin^2(\phi)\sum_{\tau_i\leq t}\Phi(\tau_i,t)^2I_d .
\end{align}
Note that $V_t=\operatorname{Cov}(X_t\mid\mathcal G_t)$ is different from the covariance $\Sigma_t=\operatorname{Cov}(X_t\mid\mathcal{F}^0_t)$. 
First consider the case with no jump, i.e. $J^X(0,t)=0$. Then $\Phi(s,t) = e^{-a(t-s)}$ 
for every $0\leq s\leq t$, and
\begin{align}\label{tso_eqn_2.32}
V_t &= e^{-2at}\Sigma_0 + \frac{\sigma^2}{2a} \left(1-e^{-2at}\right)I_d \geq \Big( e^{-2at}\sigma_0^2 + \frac{\sigma^2}{2a} \left(1-e^{-2at}\right) \Big)I_d, 
\end{align}
where we have used $\Sigma_0\geq\sigma_0^2I_d$. 
Next, we consider the case that $J^X(0,t)\geq1$, and let $\tau_k\leq t$ be the most recent explorer jump time before or at $t$. Since there are no explorer jumps in $(\tau_k,t]$,
\begin{align}
\Phi(\tau_k,t) &= e^{-a(t-\tau_k)}
\end{align}
and, for $s\in[\tau_k,t]$, $\Phi(s,t) = e^{-a(t-s)}$. Therefore, we obtain
\begin{align}\label{tso_eqn_2.62}
V_t & \geq \sin^2\phi\, \sigma_J^2\Phi(\tau_k,t)^2I_d + \sigma^2\int_{\tau_k}^t\Phi(s,t)^2\der s\,I_d  \nonumber \\  &   = \Big( \sigma_J^2\sin^2\phi\,e^{-2a(t-\tau_k)} + \frac{\sigma^2}{2a} \big(1-e^{-2a(t-\tau_k)}\big) \Big) I_d.
\end{align}
From \eqref{tso_eqn_2.32} and \eqref{tso_eqn_2.62}, we get
\begin{align*}
V_t &\geq \min\left\{ \sigma_0^2, \sigma_J^2\sin^2\phi, \frac{\sigma^2}{2a} \right\}I_d \geq \hat{\sigma}^2I_d .
\end{align*}
Thus, for every realization of the explorer jump times, $V_t-\hat{\sigma}^2I_d$ 
is positive semidefinite. Hence, using Gaussian convolution identity, we have
\begin{align*}
\mathcal L(X_t\mid\mathcal G_t) &= \mathcal{N}(M_t, V_t) = \mathcal N(0,\hat{\sigma}^2I_d) * \mathcal N(M_t,V_t-\hat{\sigma}^2I_d).
\end{align*}
Define the $\mathcal{F}^0_t$-measurable probability measure $\nu_t$ by
\begin{align}
\nu_t(A) &:= \mathbb E\left[ \mathcal N(M_t,V_t-\hat{\sigma}^2I_d)(A) \;\middle|\; \mathcal{F}^0_t \right], \qquad A\in\mathcal B(\mathbb R^d).
\end{align}
Then, for every Borel set $A\subset\mathbb R^d$, the tower property and Fubini's theorem give
\begin{align}
\begin{aligned}
\mu_t(A) &= \mathbb P(X_t\in A\mid\mathcal{F}^0_t) = \mathbb E\left[ \mathbb P(X_t\in A\mid\mathcal G_t) \;\middle|\; \mathcal{F}^0_t \right] \\
&= \mathbb E\left[ \left( \mathcal N(0,\hat{\sigma}^2I_d) * \mathcal N(M_t,V_t-\hat{\sigma}^2I_d) \right)(A) \;\middle|\; \mathcal{F}^0_t \right] = \left( \mathcal N(0,\hat{\sigma}^2I_d) * \nu_t \right)(A).
\end{aligned}
\end{align}
Therefore
\begin{align}
\mu_t &= \mathcal N(0,\hat{\sigma}^2I_d) * \nu_t .
\end{align}

\section{Steady state and Laplace approximation}\label{tso_steady_state_subsec} 
In Subsection~\ref{tso_steady_state_Sec}, we identify the steady state of mean-field SDEs and in  Subsection~\ref{tso_sec_quant_laplace_approx}, we prove the convergence of the steady state position of hunter to the global minimum with optimal rate. 

\subsection{Steady state} \label{tso_steady_state_Sec}
We begin by imposing a variance-matching condition, and subsequently analyze the system when this condition is relaxed. Let us assume 
\begin{align}\label{tso_match_condn}
 \frac{\sigma^2}{2(\eta_1 + \eta_2)} = \sigma_J^2 .
\end{align} 
 With this condition, we claim that a stationary state of the mean-field TSO dynamics is given by
\begin{align}
    \mu_\star
    =  \mathcal{N}(m_\star,\sigma_\star^2 I_d),\quad 
    Y_\star
    = m_\star,
\end{align}
where $m_\star =     \mathfrak{m}^\alpha(\mu_\star)$ and $\sigma_\star $ in steady state is $\sigma_J$ if \eqref{tso_match_condn} holds.  First, since $Y_\star=m_\star$
 and $\mathfrak{m}^\alpha(\mu_\star)=m_\star$, we have $\mathrm{M}^\alpha(\mu_\star,Y_\star)  =     m_\star$. 
 Denote the Gaussian density by
 \begin{align}
    \rho_\star(x)
    =
    \frac{1}{(2\pi\sigma_\star^2)^{d/2}}
    \exp\left(
    -\frac{|x-m_\star|^2}{2\sigma_\star^2}
    \right).
\end{align}
Since $\sigma_\star^2 \nabla_x\rho_\star(x) = -(x- m_\star)\rho_\star(x)$, we have $\nabla_x\cdot((x-m_\star)\rho_\star)
= -\sigma_\star^2\Delta_x\rho_\star$. It follows that
\begin{align}
(\eta_1 + \eta_2)\nabla_x\cdot(
(x- m_\star)\rho_\star) +
\frac{\sigma^2}{2}\Delta_x\rho_\star
&= \left(
-(\eta_1 + \eta_2)\sigma_\star^2+\frac{\sigma^2}{2}
\right)
\Delta_x\rho_\star = 0, \label{tso_eqn_2.21}
\end{align}
because $\sigma^2=2(\eta_1 + \eta_2)\sigma_\star^2$.  Let $    X\sim \mu_\star=\mathcal N(m_\star,\sigma_\star^2I_d)$, and $Z\sim \mathcal{N}(0,I_d)$, 
with $X$ and $Z$ independent. Since $\mathrm{M}^\alpha(\mu_\star,Y_\star)=m_\star$, we have
\begin{align}
    \mathcal J_{\mu_\star,Y_\star}(X,Z) =
    cX+(1-c)m_\star +
    \sqrt{1-c^2}\,\sigma_\star Z,
\end{align}
where $c:= \cos(\phi)$. Thus $\mathbb E[\mathcal J_{\mu_\star,Y_\star}(X,Z)] =
    cm_\star+(1-c)m_\star  = m_\star$ and $   \operatorname{Cov}
    \bigl( \mathcal J_{\mu_\star,Y_\star}(X,Z)\bigr)=     c^2\sigma_\star^2I_d +     (1-c^2)\sigma_\star^2I_d =     \sigma_\star^2I_d$. 
Since $\mathcal J_{\mu_\star,Y_\star}(X,Z)$ is Gaussian, we obtain $\mathcal J_{\mu_\star,Y_\star}(X,Z)
    \sim
    \mathcal N(m_\star,\sigma_\star^2I_d) =  \mu_\star$. 
Consequently,
\begin{align}
    \lambda_J \left(\mathcal J_{\mu_\star,Y_\star\#}
    (\mu_\star\otimes\nu) -  \mu_\star
    \right)  =  0. \label{tso_eqn_2.23}
\end{align}
Combining \eqref{tso_eqn_2.21} and \eqref{tso_eqn_2.23}, we get
\begin{align}
    \mathcal{A}_{\mu_\star} \mu_\star = 0. 
\end{align}
For the hunter, at the proposed steady state, the coefficients vanish.  Conditionally, the corresponding joint law is
\begin{align}
    \mathcal L(X_\star,Y_\star\mid\mathcal F_\star^0)
    =
    \mu_\star\otimes\delta_{m_\star},
\end{align}
where $\mathcal{F}_\star^0 = \mathrm{\sigma}(\bigcup_{t \geq 0}\mathcal{F}_t^0)$.

When the condition \eqref{tso_match_condn} is not satisfied, the disparity between continuous diffusion and the discrete affine jumps prevents the conditional steady state from being Gaussian.
Nevertheless, we can still explicitly determine the conditional variance. The conditional covariance $\Sigma_t = \operatorname{Cov}(X_t \mid \mathcal{F}_t^0)$ satisfies a closed equation up to a rank one term as we will demonstrate below.
To derive the dynamics of the conditional covariance $\Sigma_t $, we define  $\mathrm{D}_t := X_t - \bar{X}_t$, where $\bar{X}_t = \mathbb{E}[X_t \mid \mathcal{F}_t^0]$ is the conditional mean. Therefore, we have
\begin{align*}
\der \mathrm{D}_t 
&= \der X_t - \der \bar{X}_t \\
&= \bigl( -\eta_1(X_t - \m^\alpha(\mu_t)) - \eta_2(X_t - Y_t) \bigr) \der t + \bigl(\eta_1(\bar{X}_t - \mathfrak{m}^\alpha(\mu_t)) + \eta_2(\bar{X}_t - Y_t) \bigr) \der t + \sigma \der W_t \\  
 & \quad+ \int_{\mathbb{R}^d}\bigl( (1-\cos(\phi))
\big(
-X_{t^-}+ \mathrm{M}^\alpha(\mu_t,Y_t)
\big) + \sigma_J\sin(\phi)z\bigr) N(\der t, \der z) \\
& \quad - \mathbb{E}\bigg( \int_{\mathbb{R}^d} 
\Bigl( (1-\cos\phi) \bigl( -X_{t^-} + \mathrm{M}^\alpha(\mu_t,Y_t) \bigr) + \sigma_J\sin(\phi)z \Bigr) N(\der t, \der z) \;\Big|\; \mathcal{F}_t^0 \bigg).
\end{align*}
Applying Ito's lemma to  $\mathrm{D}_t \mathrm{D}_t^\top$ and taking the expectation on both sides conditional on $\mathcal{F}_t^{0}$ gives
\begin{align}
 \frac{\der \Sigma_t}{\der t} 
&= \mathbb{E}\left[ \frac{\der}{\der t}(\mathrm{D}_t \mathrm{D}_t^\top) \mid \mathcal{F}_t^0 \right] 
=-2(\eta_1+\eta_2)\Sigma_t + \sigma^2 I_d + \lambda_J \big( (\cos^2(\phi) - 1)\Sigma_t + \sin^2(\phi)\sigma_J^2 I_d \big)\nonumber 
\\ & \quad + \lambda_J (1 - \cos(\phi))^2  (\mathrm{M}^\alpha(\mu_t, Y_t) -\bar{X}_t)(\mathrm{M}^\alpha(\mu_t, Y_t) -\bar{X}_t)^\top, 
\label{tso_eqn_covi}
\end{align}
where we have used the following: \begin{align}
&\mathbb{E}\left[ \int_{\mathbb{R}^d} 
\Bigl( (1-\cos\phi) \bigl( -X_{t^-} + \mathrm{M}^\alpha(\mu_t,Y_t) \bigr) + \sigma_J\sin(\phi)z \Bigr) N(\der t, \der z) \;\bigg|\; \mathcal{F}_t^0 \right] \nonumber \\
&\; = \mathbb{E}\left[ \int_{\mathbb{R}^d} 
\Bigl( (1-\cos\phi) \bigl( -X_{t^-} + \mathrm{M}^\alpha(\mu_t,Y_t) \bigr) + \sigma_J\sin(\phi)z \Bigr) \lambda_J \nu(\der z) \der t \;\bigg|\; \mathcal{F}_t^0 \right] \nonumber \\
&\; = \lambda_J \mathbb{E}\big( (1-\cos\phi) \bigl( -X_{t^-} + \mathrm{M}^\alpha(\mu_t,Y_t) \bigr) \;\big|\; \mathcal{F}_t^0 \big) \der t = \lambda_J (1-\cos\phi) \bigl( -\bar{X}_{t} + \mathrm{M}^\alpha(\mu_t,Y_t) \bigr) \der t.
\end{align}
At steady state $Y = \m^\alpha(\mu)$, therefore setting $\frac{\der \Sigma_t}{\der t} = 0$ in \eqref{tso_eqn_covi} gives
\begin{equation}
 -\bigl(2(\eta_1+\eta_2) + \lambda_J(1-\cos^2\phi)\bigr) \Sigma_\star + \bigl(\sigma^2 + \lambda_J(1-\cos^2\phi)\sigma_J^2\bigr)I_d = 0.
\label{eq:var_ODE}
\end{equation}
Therefore, we arrive at the steady-state conditional covariance which we also denote by $\sigma_\star^2 I_d$  where $\sigma_\star$ is given by
\begin{equation}
\sigma_\star^2 = \frac{\sigma^2 + \lambda_J \sin^2\phi\,\sigma_J^2}{2(\eta_1+\eta_2) + \lambda_J \sin^2\phi}.
\label{eq:steady_var}
\end{equation}

\subsection{Quantitative Laplace approximation of steady state}\label{tso_sec_quant_laplace_approx}
 In \cite{carrillo2018analytical}, Laplace principle is applied directly to the initial distribution, assuming it has sufficient mass near the global minimizer. In \cite{fornasier_klock_riedl2024consensus}, the authors established a  quantitative Laplace principle valid for time-varying measures provided they have sufficient mass in the region containing $x_{\min}$. This bound has been used in analyzing additive noise CBO in \cite{additive_CBO_2026}, where it is again applied to the time-varying measure $\rho_t$. However, because the quantitative bound is obtained for arbitrary measures, the measure-splitting techniques used in \cite{fornasier_klock_riedl2024consensus} yield sub-optimal error terms. 
 
 In this subsection, by imposing local smoothness ($C^4$ in the neighborhood of $x_{\min}$) and non-degeneracy at the global minimizer, we apply the classical Laplace method directly to the self-consistent stationary equation (see \eqref{tso_eqn_3.15}). This allows us to exploit the symmetry of the Gaussian stationary measure and get explicit rate of $\mathcal{O}(1/\alpha)$ for the distance between the steady-state of treasure hunter and the global minimum.

We impose the following assumption on $f$.

\begin{assumption}\label{tso_assum_2}
   Let $f \in C^{4}(B_r(x_{\min}))$ for some $r > 0$ with $x_{\min}$ being the unique global minimizer. We assume $\nabla^2 f(x_{\min}) > 0$. We assume that there exist constants $\kappa > 0$ and  $\ell \geq 0 $ such that 
   \begin{align}\label{tso_eqn_coerc_assum}
       f(x) \geq \kappa |x|^2 - \ell, \quad  x \in \mathbb{R}^{d}.
   \end{align}
\end{assumption}

\begin{theorem}[Quantitative Laplace principle for the steady state]
\label{thm:steady_state_laplace}
Let the variance-matching condition \eqref{tso_match_condn} hold, such that the conditional steady-state law of the explorers is Gaussian with variance $\sigma_\star^2 = \sigma_J^2$, and let $c_0 = 1/(2\sigma_\star^2)$. Let the objective function $f$ satisfy Assumption~\ref{tso_assum_2}. Define the mapping $T_\alpha: \mathbb{R}^d \to \mathbb{R}^d$ as
\begin{align}
    T_\alpha(m) = \frac{\int_{\R^d} x e^{-\alpha f(x)} e^{-c_0 |x - m|^2} \der x}{\int_{\mathbb{R}^d} e^{-\alpha f(x)} e^{-c_0 |x - m|^2} \der x}.
\end{align}
Then, for all sufficiently large $\alpha > 0$, there exists a self-consistent steady-state consensus point $m_\alpha \in \mathbb{R}^d$ satisfying 
\begin{align}\label{tso_eqn_3.15}
    m_\alpha = T_\alpha(m_\alpha).
\end{align}
Moreover, there exists a constant $C > 0$, independent of $\alpha$, such that the distance between $m_\alpha$ and the global minimum satisfies the bound
\begin{align}
    |m_\alpha - x_{\min}| \le \frac{C}{\alpha}.
\end{align}
\end{theorem}
\begin{proof}
Without loss of generality, we may assume $f(x_{\min}) = 0$, as shifting $f$ by a constant cancels out in the definition of $T_\alpha(m)$.

Due to \eqref{tso_eqn_coerc_assum}, the mapping $T_\alpha$ is well-defined. Moreover, if we take a sequence $m_n $ converging to $m$ then $e^{-c_0 |x - m_n|^2} \to e^{-c_0 |x - m|^2}$. Then, again due to  \eqref{tso_eqn_coerc_assum}, dominated convergence theorem can be applied to obtain $T_\alpha (m_n) \to T_{\alpha}(m) $ as $n \to \infty$. Therefore, $T_\alpha :\mathbb{R}^{d} \to \mathbb{R}^d$ is a continuous mapping.

We now prove that for sufficiently large $R$ and sufficiently large $\alpha$, the following holds:
\begin{align*}
 \big( ( m - x_{\min}) \cdot (T_\alpha(m) - x_{\min})\big) \leq |m - x_{\min}|^2, 
\end{align*}
whenever $|m - x_{\min}| \geq R$. Let $y = x - x_{\min}$ and $\varepsilon = m - x_{\min}$, then using this notation, we have
\begin{align*}
    T_{\alpha}(m) - x_{\min} &= \frac{\int_{\mathbb{R}^d} y e^{-\alpha f(x_{\min} + y)} e^{-c_0 |y - \varepsilon|^2} \der y}{\int_{\mathbb{R}^d} e^{-\alpha f(x_{\min} + y)} e^{-c_0 |y - \varepsilon|^2} \der y}, \\ 
   \text{and}\quad  \big( \varepsilon \cdot (T_{\alpha}(m) - x_{\min})\big) &= \frac{\int_{\mathbb{R}^d} (\varepsilon \cdot y) e^{-\alpha f(x_{\min} + y)} e^{-c_0 |y - \varepsilon|^2} \der y}{\int_{\mathbb{R}^d} e^{-\alpha f(x_{\min} + y)} e^{-c_0 |y - \varepsilon|^2} \der y}.
\end{align*}
Consider the integral 
\begin{align*}
    \int_{\mathbb{R}^d} (\varepsilon \cdot y) e^{-\alpha f(x_{\min} + y)} e^{-c_0 |y - \varepsilon|^2} \der y & = \int_{|y| \leq |\varepsilon|/2 } (\varepsilon \cdot y) e^{-\alpha f(x_{\min} + y)} e^{-c_0 |y - \varepsilon|^2} \der y \nonumber 
    \\   &   \quad + \int_{|y| > |\varepsilon|/2 } (\varepsilon \cdot y) e^{-\alpha f(x_{\min} + y)} e^{-c_0 |y - \varepsilon|^2} \der y.
\end{align*}
Therefore, we get
\begin{align}\label{tso_eqn_2.7}
     \big( \varepsilon \cdot (T_{\alpha}(m) - x_{\min})\big) \leq  \frac{|\varepsilon|^2}{2 } + \frac{|\varepsilon|}{Z_\alpha(\varepsilon)} \int_{|y| > |\varepsilon|/2 }  |y| e^{-\alpha f(x_{\min} + y)} e^{-c_0 |y - \varepsilon|^2} \der y,
\end{align}
where $Z_\alpha(\varepsilon) := \int_{\mathbb{R}^d} e^{-\alpha f(x_{\min} + y)} e^{-c_0 |y - \varepsilon|^2} \der y $.

For subsequent analysis, we need a lower bound for the denominator $Z_\alpha(\varepsilon)$. Since $f \in C^4$ near $x_{\min}$, $\nabla f(x_{\min}) = 0$, and $\nabla^2 f(x_{\min}) > 0$, there are constants $\hat{r} > 0$ and $ \hat{C} > 0$ such that
\begin{equation}\label{tso_eqn_2.8}
f(x_{\min} + y) \leq \hat C |y|^2 
\end{equation}
whenever $|y| \leq \hat{r}$. For sufficiently large $\alpha$, $|y| \le \alpha^{-1/2}$ implies $|y| \le \hat r$, which in turn guarantees
\begin{equation}\label{tso_eqn_2.9}
\alpha f(x_{\min} + y) \leq \hat{C}. 
\end{equation}
On the ball $|y| \leq \alpha^{-1/2}$, we also have 
\begin{equation} \label{tso_eqn_2.10}
e^{-c_0|y-\varepsilon|^2} \geq \exp\left(-c_0|\varepsilon|^2 - 2c_0|\varepsilon|\alpha^{-1/2} - c_0\alpha^{-1}\right),
\end{equation}
since $|y - \varepsilon|^{2} \leq |\varepsilon|^{2} + 2|\varepsilon||y| + |y|^2 \leq |\varepsilon|^2 + 2|\varepsilon|\alpha^{-1/2} + \alpha^{-1}$. Therefore, using \eqref{tso_eqn_2.9} and \eqref{tso_eqn_2.10}, we obtain
\begin{align}\label{tso_eqn_2.11}
Z_\alpha(\varepsilon) \geq C\alpha^{-d/2} \exp\left(- c_0|\varepsilon|^2 - 2c_0|\varepsilon|\alpha^{-1/2} - c_0\alpha^{-1}\right),
\end{align}
where $C>0$ is independent of $\alpha$.  We now obtain a bound on the integral in the second term on right hand side of inequality \eqref{tso_eqn_2.7}. 
Since $|x|^2 = |x_{\min} + y|^2 \geq \frac{1}{2}|y|^2 - |x_{\min}|^2$, we have $f(x_{\min} + y) \geq \kappa |x_{\min} + y|^2 - \ell \ge \frac{\kappa}{2}|y|^2 - \hat{\ell}$ where $\hat{\ell} = \kappa |x_{\min}|^2 + \ell$. 
Therefore,
\begin{align}\label{tso_eqn_2.12}
 \int_{|y| > |\varepsilon| /2} |y| e^{-\alpha f(x_{\min} + y)} e^{-c_0|y - \varepsilon|^2} \der y \le C(1 + |\varepsilon|) \exp\left(\alpha \hat{\ell} - \frac{\alpha \kappa}{8}|\varepsilon|^2\right), 
\end{align} 
since $ \int_{\mathbb{R}^d} |y| e^{-c_0|y - \varepsilon|^2} \der y \le C(1 + |\varepsilon|)$ and $e^{-\alpha f(x_{\min} + y)} \leq e^{\alpha \hat{\ell}} e^{-\alpha \frac{\kappa}{2}| y|^2 }$.
Combining \eqref{tso_eqn_2.12} with \eqref{tso_eqn_2.11} gives
\begin{align*} 
\frac{|\varepsilon|}{Z_\alpha(\varepsilon)}  \int_{|y| > |\varepsilon| /2} |y| e^{-\alpha f(x_{\min} + y)} e^{-c_0|y - \varepsilon|^2} \der y \leq C|\varepsilon|(1 + |\varepsilon|) \alpha^{d/2} \exp\left(\alpha \hat \ell - \big(\frac{\alpha \kappa}{8} - c_0\big)|\varepsilon|^2 + C|\varepsilon|\alpha^{-1/2}\right).
\end{align*}
Choose $R > 0$ large enough that $\frac{\kappa}{8}R^2 > 2\hat{\ell}$. This gives 
\begin{align*} 
\alpha \hat{\ell} - \frac{\alpha \kappa}{8}R^2 \leq -\frac{\alpha \kappa}{16}R^2. 
\end{align*}
Therefore, for $|\varepsilon| \geq R$, $C |\varepsilon|(1+|\varepsilon|)\alpha^{d/2} \exp\left(-\frac{\alpha \kappa}{16}|\varepsilon|^2 + c_0|\varepsilon|^2 + 2 c_0 |\varepsilon|\alpha^{-1/2}\right) \to 0 $
as $\alpha \to \infty$. Hence, with sufficiently large $\alpha$, we can choose $\theta \in (0, 1/2)$ such that 
\begin{align*} 
\frac{|\varepsilon|}{Z_\alpha(\varepsilon)} \int_{|y| > |\varepsilon| /2} |y| e^{-\alpha f(x_{\min} + y)} e^{-c_0|y - \varepsilon|^2} \der y \leq (1-2\theta)\frac{|\varepsilon|^2}{2}
\end{align*}
for $|\varepsilon| \geq R$. Therefore,
\begin{align*}
\big(\varepsilon \cdot (T_\alpha(m) - x_{\min}) \big)\leq \frac{|\varepsilon|^2}{2} + (1 -2 \theta)\frac{|\varepsilon|^2}{2} = (1 -\theta)|\varepsilon|^2.
\end{align*}
Since $\varepsilon = m - x_{\min}$, this proves that for sufficiently large $R$ and sufficiently large $\alpha$,
\begin{equation} \label{tso_eqn_2.17}
\big( (m - x_{\min} ) \cdot (T_\alpha(m) - x_{\min})\big) \leq (1-\theta)|m - x_{\min}|^2,
\end{equation}
whenever $ |m - x_{\min}| \geq R$. Here $\theta \in (0, 1/2)$. 

We now use this estimate to obtain a fixed point. Let
$P:\mathbb R^d\to \bar{B}_R(x_{\min})$ be the radial projection onto
the closed ball centered at $x_{\min}$:
\begin{equation*}
    P(z)
    =
    \begin{cases}
    z, & |z-x_{\min}|\le R,\\
    x_{\min}     +     R\dfrac{z-x_{\min}}{|z-x_{\min}|},
    & |z-x_{\min}|>R.
    \end{cases}
\end{equation*}
Then $ P\circ T_\alpha:
    \bar{B}_R(x_{\min})
    \to     \bar{B}_R(x_{\min}) $ is continuous. By Brouwer's fixed-point theorem, there exists
$m_\alpha\in \bar{B}_R(x_{\min})$ such that
\begin{align*}
    m_\alpha=P(T_\alpha(m_\alpha)).
\end{align*}
We show  that $m_\alpha$ is indeed a fixed point of $T_\alpha$.
First consider the case when $|T_\alpha(m_\alpha)-x_{\min}|\le R$. Then, we simply have 
\begin{equation*}
    m_\alpha = P(T_\alpha(m_\alpha)) =     T_\alpha(m_\alpha).
\end{equation*}
It remains to deal with the case when $|T_\alpha(m_\alpha)-x_{\min}|>R$. 
In this case, by the definition of $P$, we have
\begin{align*}
     m_\alpha
    =  x_{\min} + R\frac{
        T_\alpha(m_\alpha)-x_{\min}}{|T_\alpha(m_\alpha)-x_{\min}|}.
\end{align*}
Hence $|m_\alpha-x_{\min}|=R$ and $m_\alpha-x_{\min}$ points in the same direction as
$T_\alpha(m_\alpha)-x_{\min}$. Equivalently, $T_\alpha(m_\alpha)-x_{\min} = \hat{a}(m_\alpha-x_{\min})$
 with $\hat{a} = \frac{|T_\alpha(m_\alpha)-x_{\min}|}{R} >1$. 
Therefore
\begin{align*}
    \big((m_\alpha-x_{\min})\cdot
    (T_\alpha(m_\alpha)-x_{\min}) \big)  &=
    \hat a |m_\alpha-x_{\min}|^2 = \hat{a} R^2 > R^2.
\end{align*}
This contradicts the estimate already proved i.e. \eqref{tso_eqn_2.17} whenever $|m-x_{\min}|=R$.
Therefore, $|T_\alpha(m_\alpha)-x_{\min}|>R$ is impossible.  Hence 
\begin{equation*}
    |T_\alpha(m_\alpha)-x_{\min}|\leq R,
\end{equation*} and
\begin{align}
    m_\alpha=T_\alpha(m_\alpha).
\end{align}
Therefore, there exists a self-consistent solution $m_\alpha$ and $m_\alpha \in \bar{B}_{R}(x_{\min})$ due to strict inward-pointing estimate \eqref{tso_eqn_2.17} outside the ball.

We now prove that every self-consistent solution satisfies the sharp estimate
\begin{align}
    |m_\alpha-x_{\min}|\le \frac{C}{\alpha}
\end{align}
for all sufficiently large $\alpha$. Let $m_\alpha$ be any solution of the self-consistency equation $ m_\alpha=T_\alpha(m_\alpha)$.  Let $\delta_\alpha:=m_\alpha-x_{\min}$ and therefore
\begin{equation*}
    \delta_\alpha = \frac{
    \int_{\mathbb R^d} y e^{-\alpha f(x_{\min}+y)}e^{-c_0|y-\delta_\alpha|^2}\der y}{ \int_{\mathbb R^d}
    e^{-\alpha f(x_{\min}+y)}
    e^{-c_0|y-\delta_\alpha|^2}\der y}.
\end{equation*}
We now use the uniform boundedness obtained above to prove the sharp rate. Since $|\delta_\alpha|\leq R$, $b:=2c_0\delta_\alpha$
is bounded uniformly in $\alpha$, i.e.
$
    |b|\le 2c_0R.
$ We set $B := 2 c_0 R$. 
Using $|y-\delta_\alpha|^2 = |y|^2-2(y\cdot\delta_\alpha) +|\delta_\alpha|^2$, 
we get
\begin{align*}
    \delta_\alpha = \frac{ \int_{\mathbb R^d}
    y e^{-\alpha f(x_{\min}+y)}
    e^{(b\cdot y)}
    e^{-c_0|y|^2}
    \der y}{    \int_{\mathbb R^d}
    e^{-\alpha f(x_{\min}+y)}
    e^{(b\cdot y)}
    e^{-c_0|y|^2}
    \der y}.
\end{align*}
For $|b|\le 2c_0 R$, denote
\begin{align*}
    \mathrm{D}_\alpha(b)
   & :=
    \int_{\mathbb R^d}
    e^{-\alpha f(x_{\min}+y)}
    e^{(b\cdot y)}
    e^{-c_0|y|^2}
    \der y,\\ 
    \mathrm{N}_\alpha(b)
   & :=
    \int_{\mathbb R^d}
    y\,e^{-\alpha f(x_{\min}+y)}
    e^{(b\cdot y)}
    e^{-c_0|y|^2}
    \der y.
\end{align*}
We now prove
\begin{equation}
    \mathrm{D}_\alpha(b)\ge c\alpha^{-d/2}\quad \text{and}\quad
    |\mathrm{N}_\alpha(b)|\le C\alpha^{-(d+2)/2},
\end{equation}
uniformly for $|b|\le 2c_0 R$. If we choose sufficiently large $\alpha $ such that $\alpha^{-1/2} \leq \hat{r}$ , we ascertain (see \eqref{tso_eqn_2.8})
\begin{equation*}
    \alpha f(x_{\min}+y)\le \hat C,
\end{equation*}
whenever $|y| \leq \alpha^{-1/2}$. Moreover, using $|b|\le 2c_0 R$, we obtain
\begin{align}\label{tso_eqn_2.29}
    \mathrm{D}_\alpha(b)
    &\ge
    \int_{|y|\le \alpha^{-1/2}}
    e^{-\alpha f(x_{\min}+y)}
    e^{(b\cdot y)}
    e^{-c_0|y|^2} \der y     \geq     c\alpha^{-d/2}.
\end{align}

We next estimate $\mathrm{N}_\alpha(b)$. We first fix a constant $r>0$ sufficiently small. Under Assumption~\ref{tso_assum_2}, there exists $\eta>0$ such that
\begin{equation*}
    f(x_{\min}+y)\ge \eta,
    \qquad |y|\ge r.
\end{equation*}
Therefore, 
\begin{align}\label{tso_eqn_2.31}
    &\int_{|y|\ge r}
    |y|e^{-\alpha f(x_{\min}+y)}
    e^{(b\cdot y)}
    e^{-c_0|y|^2}
    \der y \leq
    e^{-\alpha\eta}
    \int_{\mathbb R^d}
    |y|e^{B|y|}e^{-c_0|y|^2}\der y
    \leq    Ce^{-\alpha\eta},
\end{align}
where $C$ is independent of $\alpha$. It remains to estimate the bound when  $|y|<r$.   Let $y = \frac{z}{\sqrt{\alpha}}$ therefore $\der y = \alpha^{-d/2} \der z $. 
We have 
\begin{align}
    \int_{|y|<r}
    y\,&e^{-\alpha f(x_{\min}+y)}
    e^{(b\cdot y)}
    e^{-c_0|y|^2}
    \der y     \nonumber \\  & =  \alpha^{-(d+1)/2}
    \int_{|z|<r\sqrt{\alpha}} z    \exp\left(
        -\alpha f\left(x_{\min}+\frac{z}{\sqrt{\alpha}}\right)+
        \frac{(b\cdot z)}{\sqrt{\alpha}} - \frac{c_0|z|^2}{\alpha}
    \right)\der z. 
    \end{align}
    Using Taylor's theorem, we get
\begin{align}
f(x_{\min}+y) = \frac{1}{2} y^\top H y + \frac{1}{6} D^3 f(x_{\min})[y,y,y] + \mathcal{O}(|y|^4),
\end{align}
where $
    H:=\nabla^2 f(x_{\min})>0$. Choose $0<\gamma<1/8$. Then, for $|z|\le \alpha^\gamma$, we have
\begin{align*}
    -\alpha f\left(x_{\min}+\frac z{\sqrt\alpha}\right) + \frac{b\cdot z}{\sqrt\alpha} -
    \frac{c_0|z|^2}{\alpha}
    &=  -\frac{1}{2} z^\top H z
    +  \frac{1}{\sqrt\alpha}A(z)
    + Q_\alpha(z),
\end{align*}
where $ A(z)
    := (b\cdot z) - \frac{1}{6} D^3f(x_{\min})[z,z,z]$, $|Q_\alpha(z)| \leq C\frac{|z|^2+|z|^4}{\alpha}$ and uniformly for $|b|\le B$, $|A(z)|\le C(1+|z|^3)$. Since $\gamma<1/8$, $\frac{1}{\sqrt\alpha}A(z) + Q_\alpha(z)$
tends to zero uniformly on $|z|\leq \alpha^\gamma$. Hence Taylor's theorem, uniformly for \(|b|\le B\), gives
\begin{equation}\label{tso_eqn_3.31}
    \exp\left(
        \frac1{\sqrt\alpha}A(z)+Q_\alpha(z)
    \right) = 1 + \frac{1}{\sqrt\alpha}A(z) +    \mathcal{O}\left(\frac{1+|z|^8}{\alpha}\right)
\end{equation}
on \(|z|\le \alpha^\gamma\). 
Since  $\int_{|z| < \alpha^\gamma} z e^{-\frac{1}{2}z^{\top}H z} \der z = 0$,  we get
\begin{equation}\label{tso_eqn_2.40}
   \bigg| \int_{|z|<\alpha^\gamma}
    ye^{-\alpha f(x_{\min}+y)}
    e^{(b\cdot y)}
    e^{-c_0|y|^2}
    \der y\bigg| \leq
    C\alpha^{-(d+2)/2},
\end{equation}
uniformly for $|b|\leq 2c_0 R$.  As in \eqref{tso_eqn_2.8}, for sufficiently small $r >0$, there is a constant $\tilde{C}$ such that
\begin{align*}
    f(x_{\min} + y) \geq \tilde{C} |y|^2,\qquad |y|\leq r.
\end{align*}
For all sufficiently large $\alpha$, since $|z|>\alpha^\gamma$, $\frac{B|z|}{\sqrt\alpha}
    \leq \frac{\tilde C}{2}|z|^2$. Therefore, for sufficiently large $\alpha$, we get
\begin{align}\label{eqn_tso_3.33}
        \alpha^{-(d+1)/2}
    \int_{\alpha^\gamma<|z|<r\sqrt\alpha}
    |z|\exp\left(-\tilde{C} |z|^2
        +
        \frac{B|z|}{\sqrt\alpha}
    \right)\der z & \leq  
    C\alpha^{-(d+1)/2}
    \int_{|z|>\alpha^\gamma}
    |z|e^{-\frac{\tilde{C}}{2}|z|^2}\,\der z  \nonumber 
    \\   &  \leq  C\alpha^{-(d+2)/2}.
\end{align}
Combining \eqref{tso_eqn_2.31}, \eqref{tso_eqn_2.40} and \eqref{eqn_tso_3.33}, we get
\begin{align}\label{tso_eqn_2.41}
    |\mathrm{N}_\alpha(b)| \leq
    C\alpha^{-(d+2)/2}.
\end{align}
Therefore, using \eqref{tso_eqn_2.29} and \eqref{tso_eqn_2.41}, we obtain
\begin{align}
 |m_\alpha - x_{\min}| =    \left|\frac{N_\alpha(b)}{D_\alpha(b)}\right|
    \leq
    \frac{
        C\alpha^{-(d+2)/2} }{        c\alpha^{-d/2}} =
    \frac{C}{\alpha}.
\end{align}
\end{proof}
The Laplace asymptotics technique used above, based on Taylor expansion and coercivity-type conditions on $f$, is classical \cite{wong2001asymptotic, tierney1986accurate, ellis1982laplace}. What is specific to our setting is that the relevant integral is not fixed but self-consistent: the Gaussian kernel in $\mathcal{T}_\alpha$ is centered at the unknown consensus point, so the asymptotic estimate only becomes effective once coupled with the fixed-point equation \eqref{tso_eqn_3.15}. This coupling is what brings the treasure hunter to within $\mathcal{O}(1/\alpha)$ of the global minimizer, and therefore the analysis does not reduce to a direct application of the Laplace method.

\section{Conditional mean ODE and macroscopic structure}\label{conditional_mean_ode_section}
The purpose of this section is to identify the macroscopic gradient structure underlying the conditional mean dynamics of the explorer population. Although the discussion is partly formal, it highlights an important connection between interacting particle methods with consensus drift and gradient-based dynamics: consensus-drift can encode a smoothed descent direction for the underlying objective landscape. To this end, let $\bar{X}_t := \mathbb{E}( X_t\, \mid \, \mathcal{F}^0_t)$. From \eqref{tso_eqn_mf}, we have
\begin{align} \label{tso_eqn_2.25}
    \der X_t &= -\eta_1(X_t - \m^\alpha(\mu_t)) \der t - \eta_2 (X_t - Y_t)\der t  \nonumber \\  &  \quad+ \lambda_J(1 - \cos(\phi)) (-X_{t^-} + \mathrm{M}^{\alpha}(\mu_t, Y_t)) \der t  + \der\mathcal{M}_t,
\end{align}
where $\mathcal{M}_t$ denotes the martingale term i.e. 
\begin{align}
    \der \mathcal{M}_t =  \sigma\der W_t + \int_{\mathbb{R}^d}\big( ( 1- \cos(\phi))\big( - X_{t^-} + \mathrm{M}^{\alpha}(\mu_t, Y_t)\big) + \sigma_J \sin(\phi)z\big) \tilde{N}(\der t, \der z) 
\end{align}
with $\tilde{N}(\der t, \der z) = N(\der t, \der z) - \lambda_J \nu(\der z) \der t  $, i.e. the compensated Poisson random measure. Taking expectation on  both sides of \eqref{tso_eqn_2.25} conditional on $\mathcal{F}_t^0$, we obtain 
\begin{align}\label{tso_eqn_2.27}
    \frac{\der}{\der t} \bar{X}_t &= -\hat{\eta}_1(\bar{X}_t - \m^\alpha(\mu_t)) - \hat{\eta}_2 (\bar{X}_t - Y_t), 
\end{align}
where $\hat{\eta}_1 := \eta_1 + \lambda_J ( 1 - \cos(\phi)) \frac{\kappa_1}{\kappa_1 + \kappa_2}$ and $ \hat{\eta}_2 := \eta_2 + \lambda_J(1- \cos(\phi))\frac{\kappa_2}{\kappa_1 + \kappa_2}$.  While \eqref{tso_eqn_2.27}   describes the dynamics of $\bar{X}_t$ with linear pull toward the global weighted average $\m^\alpha(\mu_t)$ and the local tracking by treasure hunter $Y_t$, we next discuss the underlying descent behavior. To investigate the gradient-flow structure embedded in \eqref{tso_eqn_2.27}, we need to link the spatial distribution of the explorers to objective function $f$ and finally with effective free energy as introduced in \eqref{tso_eqn_eff_free_ene} below.
Let 
\begin{align}
    Z_t^\alpha = \int_{\mathbb{R}^d} e^{-\alpha f(x) } \mu_t(\der x)
\end{align}
and define the tilted conditional law 
\begin{align}
    \pi_t^\alpha(\der x)  =  \frac{e^{-\alpha f(x)}}{Z_t^\alpha}\mu_t(\der x).
\end{align}
In this notation, we have
\begin{align}
    \m^{\alpha}(\mu_t) = \int_{\mathbb{R}^d} x \pi_t^\alpha (\der x). 
\end{align}

Let us assume that conditional law $\mu_t$ admits a Stein kernel $\mathcal{K}_t$ around its mean $\bar{X}_t$, i.e. the following holds:
\begin{align}
     \nabla_x\cdot( \mathcal{K}_t(x) \mu_t(x)) = -(x - \bar{X}_t)\mu_t(x),
\end{align}
and integrating with sufficiently fast decaying test function $g$, we get
\begin{align}
    \int_{\mathbb{R}^d}(x- \bar{X}_t) g(x) \mu_t(\der x) = \int_{\mathbb{R}^d}\mathcal{K}_t(x) \nabla g(x) \mu_t(\der x). 
\end{align}
Taking $g = e^{-\alpha f(x)}$, we have
\begin{align}
 \int_{\mathbb{R}^d}(x- \bar{X}_t) e^{-\alpha f(x)} \mu_t(\der x) = -\alpha \int_{\mathbb{R}^d} \mathcal{K}_t(x) \nabla f (x) e^{-\alpha f(x)} \mu_t(\der x). 
\end{align}
Dividing by $Z_t^\alpha$ gives
\begin{align}\label{tso_eqn_3.11}
    -( \bar{X}_t - \m^{\alpha}(\mu_t)) = -\alpha \mathbb{E}_{\pi_t^\alpha}\big( \mathcal{K}_t(X) \nabla f(X)  \big).
\end{align}
Therefore, using \eqref{tso_eqn_3.11}  in \eqref{tso_eqn_2.27}, we get 
\begin{align}\label{tso_eqn_3.12}
\frac{\der}{\der t} \bar{X}_t &= -\hat{\eta}_1\alpha \mathbb{E}_{\pi_t^\alpha}\big( \mathcal{K}_t(X) \nabla f(X)  \big) - \hat{\eta}_2 (\bar{X}_t - Y_t).
\end{align}
The above calculations show that the dynamics of $\bar{X}_t$,  excluding the effects of treasure hunter's dynamics, follows a weighted averaged, Stein preconditioned gradient direction.

The Stein-kernel identity above provides a useful formal link between the conditional mean dynamics and the gradient of the original objective $f$, but by itself it does not provide a closed macroscopic descent, and an actual descent interpretation would additionally require positive-definiteness of kernel. However, consensus drift driven methods are inherently designed to bypass local traps.  We therefore turn to a different viewpoint based on free energy and Gaussian smoothing, which reveals how consensus drift induces descent for an effective free energy rather than directly for $f$. We begin by defining the free energy of a Gaussian cloud centered at $p$
\begin{align}\label{tso_eqn_eff_free_ene}
    F_{\alpha,\sigma_\star^2}(p) = -\frac{1}{\alpha} \log \int_{\mathbb{R}^d} e^{-\alpha f(x)} \varphi_{\sigma_\star^2}(x-p)\der x,
\end{align}
where $\varphi_{\sigma_\star^2}$ is the density of $\mathcal{N}(0, \sigma_\star^2 I_d)$.  We write the exact Gibbs barycenter of this Gaussian cloud centered at $p$,
\begin{equation}
    \mathsf{m}^\alpha_\varphi(p) :=  \frac{\int_{\mathbb{R}^d} x e^{-\alpha f(x)} \varphi_{\sigma_\star^2}(x-p)\der x}{\int_{\mathbb{R}^d} e^{-\alpha f(x)} \varphi_{\sigma_\star^2}(x-p
    )\der x}.
\end{equation}
Differentiating $F_{\alpha,\sigma_\star^2}(p)$ with respect to $p$, and using the  identity $\nabla_p \varphi_{\sigma_\star^2}(x-p) = \frac{x-p}{\sigma_\star^2} \varphi_{\sigma_\star^2}(x-p)$, we obtain
\begin{equation}
    \nabla_p F_{\alpha,\sigma_\star^2}(p) = \frac{p - \mathsf{m}^\alpha_\varphi(p)}{\alpha\sigma_\star^2}.
\end{equation}
Rearranging yields the following:
\begin{equation}
    \mathsf{m}^\alpha_\varphi(p) - p = -\alpha\sigma_\star^2 \nabla F_{\alpha,\sigma_\star^2}(p).
    \label{tso_eqn_barycenter_gradient}
\end{equation}
Using \eqref{tso_eqn_barycenter_gradient} in \eqref{tso_eqn_2.27}, we get
\begin{align}\label{tso_ode_mf}
     \frac{\der}{\der t} \bar{X}_t &= -\hat{\eta}_1\alpha\sigma_\star^2 \nabla_{\bar{x}} F_{\alpha,\sigma_\star^2}(\bar{X}_t) - \hat{\eta}_2 (\bar{X}_t - Y_t) + \hat{\eta}_1(\m^\alpha(\mu_t) - \mathsf{m}^{\alpha}_\varphi(\bar{X}_t)).
\end{align}
Equation~\eqref{tso_ode_mf} shows
how the hunter $Y_t$ alters the average dynamics. Because $Y_t$ is updated via a Poisson mechanism that strictly accepts historically better objective values ($f(Y_t) \le f(Y_{t^-})$), it acts as a monotonically improving, global memory state.  The term $(Y_t - \bar{X}_t)$ injects an adaptive, strongly convex restorative force into the macroscopic ODE. Rather than following the time-varying weighted averaging $\m^\alpha$, the swarm is continuously anchored to $Y_t$. This breaks the any reversibility of the standard gradient flow even if $\mu_t$ belongs to Gaussian manifold (i.e. third term in \eqref{tso_ode_mf} is zero). 

Let us define the  effective potential $\mathcal{V}(x; Y_t)$:
\begin{equation}
    \mathcal{V}(x\,;\, Y_t) = \hat{\eta}_1 \alpha \sigma_\star^2 F_{\alpha,\sigma_\star^2}(x) + \frac{\hat{\eta}_2}{2}|x - Y_t|^2.
\end{equation}
The ODE \eqref{tso_ode_mf} for the swarm's time-varying center can then be written as:
\begin{equation}\label{tso_ode_grad_flow}
    \frac{\der \bar{X}_t}{\der t} = -\nabla_x \mathcal{V}(\bar{X}_t\,;\, Y_t) +  \hat{\eta}_1(\m^\alpha(\mu_t) - \mathsf{m}^{\alpha}_\varphi(\bar{X}_t)).
\end{equation} 
A major simplification happens if we assume that $\mu_0 = \mathcal{N}(\bar{X}_0, \sigma_\star^2)$, the variance matching condition \eqref{tso_match_condn} holds and for a transient duration of time the following is true 
\begin{align}\label{tso_eqn_transient_condi}
    \kappa_1 ( \m^\alpha(\mu_t) -  \bar{X}_t) + \kappa_2 (Y_t - \bar{X}_t) = 0. 
\end{align}
 In this case, ODE \eqref{tso_ode_grad_flow} simply reduces to gradient flow:
 \begin{align}
      \frac{\der \bar{X}_t}{\der t} = -\nabla_{\bar{x}} \mathcal{V}(\bar{X}_t\,;\, Y_t).
 \end{align}
 Note that the above gradient flow interpretation is only true for transient duration of time when alignment condition \eqref{tso_eqn_transient_condi} holds  and therefore for not all $t$. Although \eqref{tso_match_condn} is restrictive and  \eqref{tso_eqn_transient_condi} is transient, this allows us to gain insights into the TSO dynamics and thus establish  a connection between explorers' dynamics and the macroscopic optimization landscape.

Based on the discussion above, we infer that, at the microscopic level, individual explorers traverse highly non-convex terrain. However, at the macroscopic level, the swarm performs a continuous-time gradient descent. This macroscopic viewpoint is visualized in Figure~\ref{fig_from_eyes_of_swarm} which illustrates the smoothed free-energy landscape seen by the swarm, in contrast to the original rugged objective landscape explored by individual particles.
\begin{figure}[ht]
     \centering
     \includegraphics[width=1\linewidth]{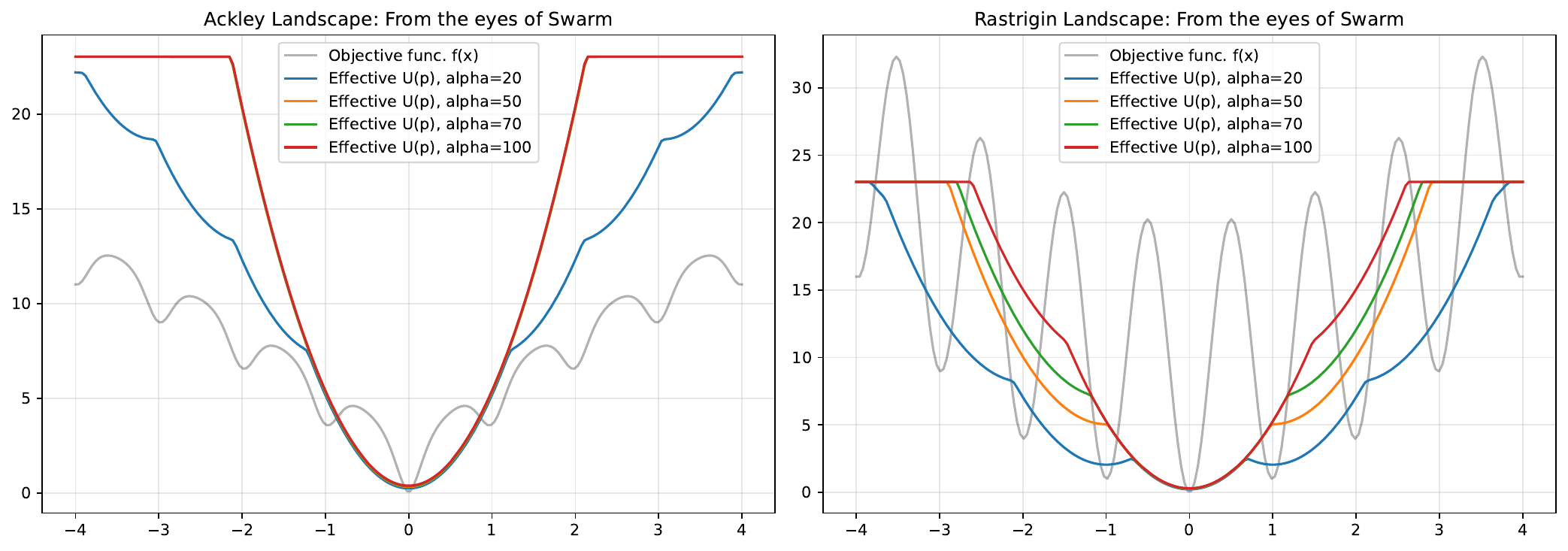}
     \caption{\textit{From the eyes of swarm.} Macroscopic potential induced by Gaussian swarm smoothing for the one-dimensional Ackley and Rastrigin objectives. Here, $U(p) = \sigma_\star^2\alpha F_{\alpha,\sigma_\star^2}(p)$ (see \eqref{tso_eqn_barycenter_gradient}) with $\sigma_\star = 0.1$. The black curves show the original objective functions, while the colored curves show the corresponding smoothed potentials $U$ depending on free energy $F_{\alpha, \sigma_\star^2}$. The figure illustrates that the swarm evolves according to a Gibbs-smoothed macroscopic landscape rather than the original microscopic objective.}
     \label{fig_from_eyes_of_swarm}
 \end{figure}

 \section{Uncertainty quantification via TSO}
\label{tso_sec_uq}
In this section, we describe how the TSO dynamics can be used as a building block for uncertainty quantification in inverse problems. The main point is that, under the variance-matching condition \eqref{tso_match_condn} introduced above, the equilibrium explorer cloud has an explicit Gaussian structure. However, this Gaussian covariance is determined by algorithmic parameters and therefore cannot, by itself, be interpreted as posterior uncertainty. Instead, we view the TSO explorers' cloud  as a reference ensemble which represents the basin of interest of the objective function. Therefore, a subsequent geometry-calibration step is needed to transform this reference ensemble into a cloud whose covariance reflects the local geometry of the inverse problem. We consider an inverse problem of the form
\begin{align}
y = \mathcal{G}(x) + \zeta,
\qquad
\zeta\sim \mathcal N(0,\Gamma),
\end{align}
where $x\in\mathbb R^d$ is the unknown parameter, $y\in\mathbb R^m$ is the observed data, $\mathcal G:\mathbb R^d\to\mathbb R^m$ is the forward map, and $\Gamma$ is the observational noise covariance. If a Gaussian prior $x\sim\mathcal N(x_0,C_0)$ is imposed, then the negative log-posterior, up to an additive constant, is
\begin{align}
\Phi(x) :=
\frac{1}{2}
\big|\Gamma^{-1/2}\bigl(\mathcal G(x)-y\bigr)\big|^2
+ \frac12 \big|C_0^{-1/2}(x-x_0)\big|^2 .
\label{eq:tso_uq_phi}
\end{align}
In this setting, the optimization target in TSO is $f=\Phi$. For large values of  $\alpha$, the Gibbs-weighted average $\mathfrak m^\alpha$ amplifies regions of small objective value, and the treasure hunter is designed to capture such low-objective regions. Thus, in the inverse-problem setting, TSO may be interpreted as a derivative-free mechanism for locating a MAP-type region of the posterior distribution.

As discussed earlier, the steady state of coupled dynamics is a self-consistent equilibrium satisfying $Y_\star=m_\star$ and
\begin{align}
m_\star =
\m^\alpha
\bigl(
\mathcal N(m_\star,\sigma_\star^2 I_d)\bigr),
\label{eq:tso_self_consistency_uq}
\end{align}
the equilibrium explorer law is
\begin{align}
\mu_\star^{\rm TSO}
=
\mathcal N(m_\star,\sigma_\star^2 I_d).
\label{eq:tso_raw_gaussian_uq}
\end{align}
The covariance $\sigma_\star^2 I_d$ is a parameter in the method chosen to keep a balance of exploration and exploitation. Therefore, it does not encode the anisotropic directions of uncertainty associated with the posterior distribution. In particular, in an ill-posed inverse problem, the posterior may be sharply concentrated along data-informed directions and diffuse along weakly observed or prior-dominated directions. Therefore, we use TSO cloud as a localized reference measure and utilize a transport map which pushes it to a geometry-aware approximation of the posterior. This separates the roles of the algorithm: TSO performs global optimization and basin capture, while the post-processing map performs local uncertainty calibration. For this purpose, local Laplace and Gauss-Newton covariance approximations of the posterior are used in Bayesian inverse problems (see \cite{crestel2017optimal}). The derivative-free ensemble covariance calibration is in the spirit of ensemble Kalman methods for inverse problems~\cite{iglesias2013ensemble}, while the optimal transport interpretation mentioned below is related to ensemble transforms~\cite{reich2013nonparametric}.

\paragraph{Laplace-calibrated TSO.}
Suppose first that sufficient derivative information is available near $m_\star$. If $\Phi$ is twice differentiable and $m_\star$ is close to a nondegenerate local minimizer, then the local Laplace approximation of the posterior is
\begin{align}
\pi_{\rm Lap} \approx \mathcal N(m_\star,\Sigma_{\rm Lap}),
\qquad \Sigma_{\rm Lap} = \bigl(\nabla^2\Phi(m_\star)\bigr)^{-1},
\label{eq:tso_laplace_covariance}
\end{align}
provided $\nabla^2\Phi(m_\star)$ is positive definite. In inverse problems, one often replaces the exact Hessian by the Gauss-Newton Hessian
\begin{align}
H_{\rm GN}(m_\star)
=
J(m_\star)^\top\Gamma^{-1}J(m_\star)
+
C_0^{-1},
\label{eq:tso_gn_hessian}
\end{align}
where $J(m_\star)=D\mathcal G(m_\star)$. The corresponding covariance is $\Sigma_{\rm GN}
= H_{\rm GN}(m_\star)^{-1}$. 
Given the TSO equilibrium cloud $X^{i,N}\sim\mathcal N(m_\star,\sigma_\star^2 I_d)$, define the affine map
\begin{align}
T_{\rm Lap}(x)
=
m_\star
+
\Sigma_{\rm Lap}^{1/2}
(\sigma_\star^2 I_d)^{-1/2}
(x-m_\star).
\label{eq:tso_laplace_transport}
\end{align}
Then $(T_{\rm Lap})_\#\mathcal N(m_\star,\sigma_\star^2 I_d) = \mathcal N(m_\star,\Sigma_{\rm Lap})$. Similarly, if the Gauss-Newton covariance is used, one obtains $T_{\rm GN}(x) = m_\star + \Sigma_{\rm GN}^{1/2}(\sigma_\star^2 I_d)^{-1/2}(x-m_\star)$.

\paragraph{Derivative-free ensemble calibration.}
Let
$R(x):=\bigl(\Gamma^{-1/2}(\mathcal{G}(x)-y),\,
\Gamma_0^{-1/2}(x-m_0)\bigr)^{\top}\in\mathbb{R}^{m+d}$,
so that $\Phi(x)=\frac{1}{2}| R(x)|^2$. At calibration step
$\ell$, set $R_\ell^{i,N}:=R(X_\ell^{i,N})$,
$\bar{X}_\ell:=N^{-1}\sum_{i=1}^N X_\ell^{i,N}$, and
$\bar{R}_\ell:=N^{-1}\sum_{i=1}^N R_\ell^{i,N}$. Define
$C_\ell^{XX}:=(N-1)^{-1}\sum_{i=1}^N
(X_\ell^{i,N}-\bar{X}_\ell)
(X_\ell^{i,N}-\bar{X}_\ell)^\top$, $C_\ell^{XR}:=(N-1)^{-1}\sum_{i=1}^N(X_\ell^{i,N}-\bar{X}_\ell)(R_\ell^{i,N}-\bar{R}_\ell)^\top$, and $C_\ell^{RR}:=(N-1)^{-1}\sum_{i=1}^N
(R_\ell^{i,N}-\bar{R}_\ell)(R_\ell^{i,N}-\bar{R}_\ell)^\top$, with $C_\ell^{RX}:=(C_\ell^{XR})^\top$. For $L$ equally tempered calibration steps, the Kalman gain is
\begin{equation}
    K_\ell =
    C_\ell^{XR}
    \bigl(C_\ell^{RR}+L I_{m+d}\bigr)^{-1}.
    \label{tsoeqn_kalman_gain}
\end{equation}
Since the target residual is zero, the ensemble mean is updated as
\begin{equation}
    \bar{X}_{\ell+1} =  \bar{X}_\ell-K_\ell\bar{R}_\ell.
    \label{eq:augmented_kalman_mean}
\end{equation}
The corresponding local Kalman covariance is
\begin{equation}
    \Sigma_{\ell+1}^{\mathrm{ens}}
    = C_\ell^{XX} -  C_\ell^{XR}\big(C_\ell^{RR}+L I_{m+d}\big)^{-1}C_\ell^{RX}.
    \label{tsoeqn_kalman_covariance}
\end{equation}
The calibrated ensemble is updated by the following deterministic affine transform:
\begin{equation}
    X_{\ell+1}^{i,N}
    =
    \bar{X}_{\ell+1}
    +
    \bigl(\Sigma_{\ell+1}^{\mathrm{ens}}\bigr)^{1/2}
    \bigl(C_\ell^{XX}\bigr)^{-1/2}
    \bigl(X_\ell^{i,N}-\bar{X}_\ell\bigr),
    \qquad i=1,\ldots,N.
    \label{tsoeq:kalman_ensemble}
\end{equation}

Here, $C_0^{XX}$ is a local reference covariance centered at the TSO hunter's position, rather than the original Bayesian prior covariance.
It is initialized either from the unweighted empirical covariance of the terminal
TSO explorers or from $\sigma_\star^2 I_d$. The prior enters explicitly
through the residual component $\Gamma_0^{-1/2}(x-x_0)$. Hence, the
calibration is interpreted as a derivative-free local Kalman
approximation of the posterior, evolving TSO explorers according to the data informed geometry of the inverse problem.

\paragraph{Optimal-transport interpretation.}
The affine maps can be interpreted as optimal transport maps from the raw TSO equilibrium measure to a calibrated posterior approximation. In the Gaussian case, the optimal transport map $T(x) = m + A(x-m)$ where $A = \sigma_\star^{-1}\Sigma_p^{1/2}$. 
 
\begin{figure}[H]
     \centering
     \includegraphics[width=1\linewidth]{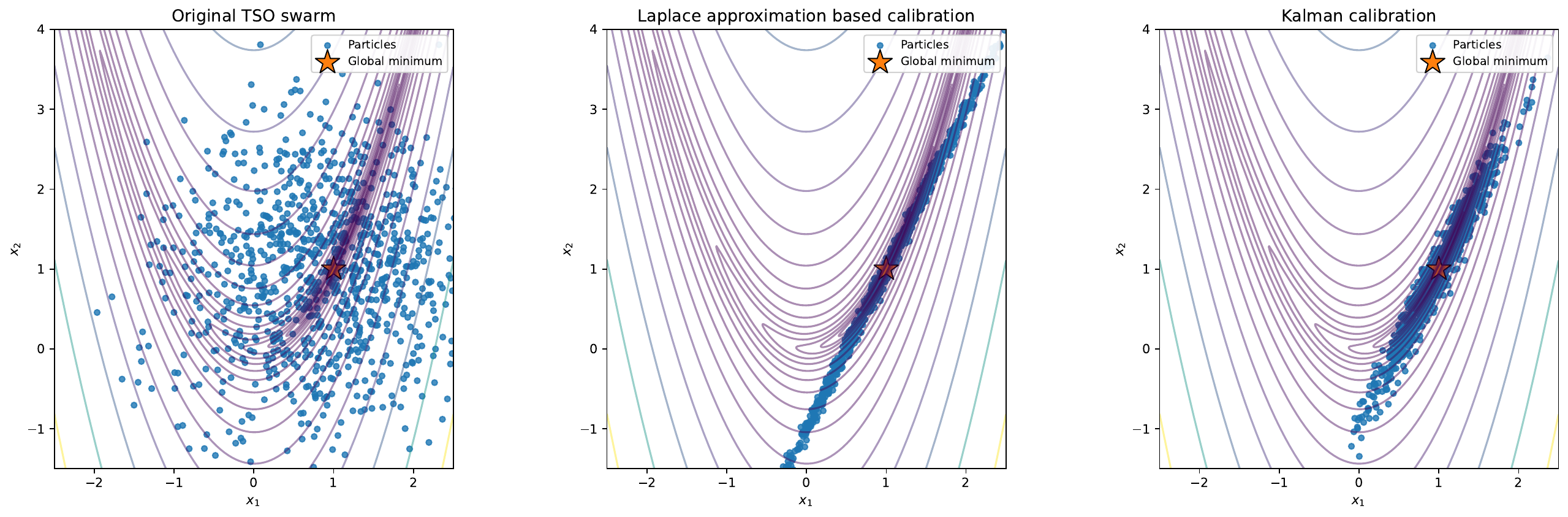}
     \caption{Rosenbrock objective $f(x_1, x_2) = (1 - x_1)^2 + 100(x_2 - x_1^2)^2 = |r(x)|^2$, with $r(x) = (x_1 - 1, 10(x_2 - x_1^2))$, global minimizer $x_\star = (1, 1)$, $N = 1000$, steady-state TSO covariance $\sigma_\star^2 I_2$ with $\sigma_\star = 1$. The left plot shows the isotropic TSO reference cloud centered at $x_\star$. The middle panel shows Laplace calibration, where the cloud is affinely transformed to have covariance $(\nabla^2 f(x_\star))^{-1}$. The right plot shows derivative-free Kalman calibration with $\Gamma = \frac{1}{2} I_2$, which reshapes the same TSO cloud. Both calibrations convert the steady cloud into an anisotropic cloud aligned with the local valley geometry.}
     \label{tso_uq}
 \end{figure}

\section{Numerical experiments}\label{numerical_exp_sec}
In Section~\ref{tso_sect_implementation}, we present the numerical implementation of TSO interacting particle system and in Section~\ref{comp_sec_cb_tso}, we compare TSO with anisotropic CBO and additive CBO on four optimization problems whose objective functions are derived from ODE constrained inverse problems.   In Section~\ref{tso_subsec:low_dimensional_experiment}, we test the performance of CBS, EKS and TSO + Kalman calibration on a Bayesian inverse problem. 

\subsection{Implementation of TSO}\label{tso_sect_implementation}
Consider the uniform partition $0=t_0<\dots<t_n = T$ of the time interval $[0,T]$ with discretization step $h$, i.e $t_{k +1} = t_k + h$. The TSO method evolves both an explorer population $ X_k^1,\dots,X_k^N$ 
and a hunter variable $Y_k\in\mathbb{R}^d$. Let $\cE^N_k$ denote the empirical measure of $N$ explorers at $k$-th iteration. For brevity, we take $\eta_2 = 0 $ and replace notation $\eta_1$ with $\eta$. Given the weighted average $\m_k : = \m^\alpha(\cE^N_k)$, define the hunter-dependent jump center
\begin{equation*}
    \mathrm M_k
    = \frac{
        \kappa_1 \m_k+\kappa_2 Y_k}{\kappa_1+\kappa_2
    }.
\end{equation*}
For each particle, draw a full Poisson jump count $    J_k^i \sim \operatorname{Poisson}(\lambda_J h)$.  Conditional on $J_k^i$, simulate independent Gaussian marks
\begin{align*}
    \zeta_{k,\ell}^i \sim \mathcal{N}(0,I_d),
    \qquad
    \ell=1,\dots,J^i_k.
\end{align*}
The explorer update is
\begin{align}
    X_{k+1}^i
    &=
 X_k^i -
    \eta h\left(X_k^i-\m_k\right) +
    \sigma\sqrt{h}\,\xi_k^i +
    \sum_{\ell=1}^{J_{k}^{i}}
    \left[    (1-\cos\phi)\left(-X_k^i+ \mathrm M_{k}\right)
        +
 \sigma_J\sin\phi\,\zeta_{k,\ell}^i
    \right] \nonumber \\  
    & = X_k^i
    - \eta h\left(X_k^i-\m_k\right)
 +  \sigma\sqrt{h}\,\xi_k^i +
    J_k^i(1-\cos\phi)\left(-X_k^i+\mathrm M_k\right)
    + \sigma_J\sin\phi
    \sum_{\ell=1}^{J_k^i}\zeta_{k,\ell}^i.
\end{align}
After updating the explorer population, the new weighted average $\m_{k+1}=\m^\alpha(\cE^N_{k +1})$
is computed. The hunter drifts towards $\m_{k+1}$:
\begin{equation}
    \hat{Y}_{k+1}
    =  Y_k - \beta h\left(Y_k- \m_{k+1}\right).
\end{equation}
Let $B_{k+1} \sim \mathrm{Bernoulli}(1-\exp(-\lambda_Y h))$ denote whether a hunter jump is attempted. Then
\begin{equation}\label{tso_eqn_5.3}
Y_{k+1} =
\begin{cases}
\m_{k+1},
& \text{if } B_{k+1}=1 \text{ and } f(\m_{k+1})<f(\hat{Y}_{k+1}),\\
\hat{Y}_{k+1},
& \text{otherwise}.
\end{cases}
\end{equation}

Therefore, a hunter jump is attempted with probability $  1-\exp(-\lambda_Y h)$. If a hunter jump occurs and the updated $\m_{k+1}$ improves the current hunter objective value, i.e., if $f(\m_{k+1}) < f(\hat{Y}_{k+1})$ then the hunter teleports to $Y_{k+1}
    =  \m_{k+1}$, otherwise, $
    Y_{k+1}
    =
    \hat{Y}_{k+1}.$

\subsection{Comparison between additive CBO, anisotropic CBO and TSO}\label{comp_sec_cb_tso}
\subsubsection{Objective functions derived via ODE-constrained problems} 
We evaluate the optimization methods on four two-dimensional
ODE-constrained inverse problems. In each example, the decision variable $    x=(x_1,x_2)\in\Theta\subset\mathbb R^2$
specifies an initial velocity, momentum, or position. For each candidate
$x$, we solve an initial value problem
\begin{equation}\label{tso_eqn_new_6.4}
    \dot z(t)=G(t,z(t);x),\qquad z(0)=z_0(x),
    \qquad t\in[0,t_f],
\end{equation}
where $G$ is a vector field varying with test problem and we compare the terminal state with a reference terminal state generated by a
known reference parameter $x_{\rm ref}$. For all test problems, the objective function has the form
\begin{align}\label{tso_eqn_new_6.5}
    f(x)=
    \frac12\left(
        |q_x(T)-q_{\rm ref}(T)|^2
        +
        \omega_v|v_x(T)-v_{\rm ref}(T)|^2
    \right) +\frac{\lambda_{\rm regu}}{2}
    |\theta-\theta_{\rm ref}|^2,
\end{align}
where $q$ denotes the position/configuration variables, $v$ denotes the
velocity or momentum variables, $\omega_v\ge0$ is a velocity-weighting
parameter, and $\lambda_{\rm regu}=0.1$ is the regularization weight.
We construct four objective functions corresponding to the following problems : (i)
Charged particle in a nonuniform electromagnetic field, (ii) Double pendulum, (iii) Magnetic confined particle, (iv) Damped particle in multi-well potential. The detailed formulation and construction of these model problems are provided in the Appendix~\ref{detal_exp_appendix}. However, our main focus is comparison of TSO with other methods, for us what is important is the geometric features/topography of these functions and therefore the corresponding contour plots are illustrated in Figure~\ref{tso_fig_contour_plots}. 

In this experiment, we compare with two other methods anisotropic CBO and additive-noise CBO. For the sake of completeness, we write below the one-step procedure of the two methods. Let $\xi_k^i \sim \mathcal{N}(0,I_d)$ be independent standard Gaussian vectors. The anisotropic CBO dynamics are implemented as
\begin{equation}
    X_{k+1}^i
    = X_k^{i}
    - h\lambda_\mathrm{aniso}\left(X^i_k- \m_k\right)
    + \sigma_{\rm aniso}
    \mathrm{Diag} (X_k^i- \m_k)
        \sqrt{h}\,\xi^i_{k},
    \qquad
    i=1,\dots,N,
\end{equation}
where $\mathrm{Diag}(a)$ denotes the diagonal matrix with diagonal $a$, and $\lambda_{\rm aniso}>0$  and  $\sigma_{\rm aniso}>0$ are the consensus drift and diffusion parameters. The additive CBO dynamics are implemented as
\begin{equation}
    X_{k+1}^{i}
    = X_k^i - h\lambda_\mathrm{add}\left(X_k^i-\m_k\right)
    +  \sigma_{\rm add}\sqrt{h}\,\xi_k^i,
    \qquad
    i=1,\dots,N,
\end{equation}
where $\lambda_{\rm add}>0$ is the consensus drift parameter and $\sigma_{\rm add}>0$ is the additive diffusion coefficient.

\begin{figure}[ht]
    \centering
    \includegraphics[width=1\linewidth, height = 0.5\textheight]{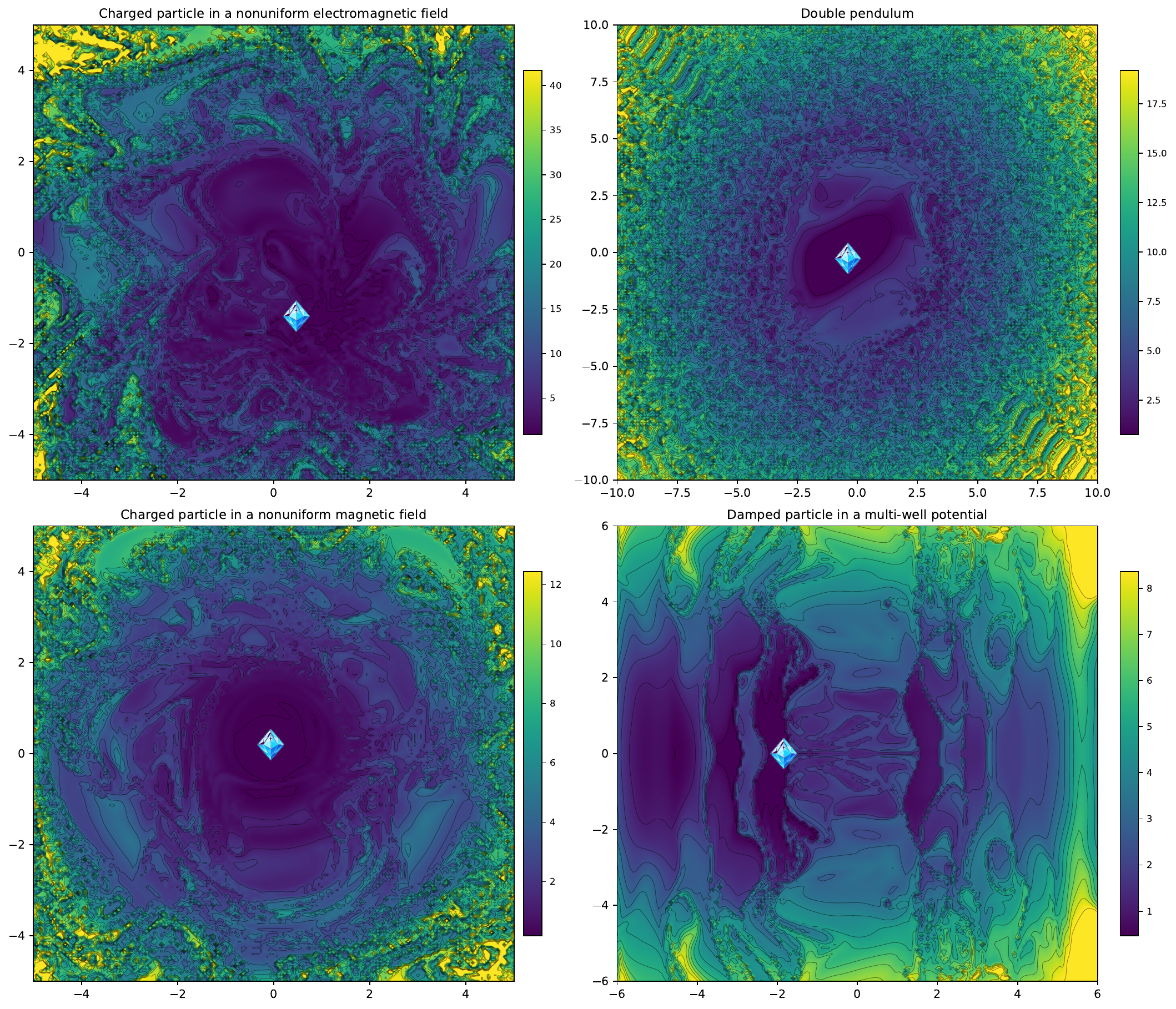}
    \caption{Contour plots of the objective functions. Diamond denotes the global minimum.}
    \label{tso_fig_contour_plots}
\end{figure}

\subsubsection{Choice of parameters}
 The exploratory capabilities of the three methods are not directly comparable, because additive CBO has persistent additive noise, anisotropic CBO has state-dependent multiplicative noise, and TSO contains both Brownian noise and compound Poisson jumps. We therefore first derive a common initial one-step spread criterion. Let $  Z_k^i = X_k^i-\m_k$ denote the deviation of a particle from the weighted-average. For the purpose of parameter calibration, we use a frozen-consensus approximation over one time step. Let
\begin{align}
    S_k = \mathbb{E}|Z_k|^2
\end{align}
denote the mean-square particle spread. For additive CBO, freezing the consensus gives $  Z_{k+1} =
    (1-\lambda_{\rm add} h)Z_k
    + \sigma_{\rm add}\sqrt{h}\xi_k$. Thus
\begin{align*}
    S_1^{\rm add} =
    (1-\lambda_{\rm add} h)^2 S_0+    d\sigma_{\rm add}^2 h.
\end{align*}
For anisotropic CBO, the frozen-consensus spread satisfies
\begin{align}
    S_1^{\rm aniso}
    =  \left[
        (1-\lambda_{\rm aniso} h)^2 + \sigma_{\rm aniso}^2 h
    \right]S_0.
\end{align}
Therefore, matching the initial one-step spread of anisotropic CBO to that of additive CBO gives $S_1^{\rm aniso}  =  S_1^{\rm add}$. If both CBO variants have $\lambda_{\rm add} = \lambda_{\rm aniso}$, then
\begin{align}\label{tso_eqn_6.9}
    \sigma_{\rm aniso}^2
    =
    \frac{d\sigma_{\rm add}^2}{S_0}, \quad \text{i.e.} \quad \sigma_{\rm aniso}
    =     \sqrt{
        \frac{d\sigma_{\rm add}^2}{S_0}}.
\end{align}
For TSO, we again use a frozen-center approximation and, for brevity, take $\mathrm{M}_k\approx \m_k$. Let $a=1-\cos\phi$ and $J_k\sim\operatorname{Poisson}(\lambda_J h)$. 
Then the explorer deviation satisfies
\begin{align*}
    Z_{k+1}
    =
    \left(1-\eta h-a J_k\right)Z_k
    + \sigma\sqrt{h}\xi_k
    +  \sigma_J\sin\phi
    \sum_{\ell=1}^{J_k}\zeta_{k,\ell}.
\end{align*}
Using the exact Poisson moments $\mathbb{E}[J_k]=\lambda_Jh $ and $\E[J_k^2]=\lambda_Jh+(\lambda_Jh)^2$, 
the TSO spread factor is
\begin{align*}
    A^{\rm TSO}
    &=
    \mathbb{E}\left[\left(1-\eta h-aJ_k\right)^2
    \right]
    = 1 - h\left(2\eta+\lambda_J\sin^2\phi\right)
    + h^2
    \left[\eta+\lambda_J(1-\cos\phi)
    \right]^2.
\end{align*}
Hence $S_1^{\rm TSO}
    =  A^{\rm TSO}S_0
    +  dh
    \left[\sigma^2
        + \lambda_J\sigma_J^2\sin^2\phi
    \right]$. 
Initial one-step matching with additive CBO gives $S_1^{\rm TSO} = S_1^{\rm add}$. 
Solving for \(\sigma_J\) yields
\begin{align}
    \sigma_J
    =\sqrt{\frac{\left[(1-\lambda_{\rm add} h)^2-A^{\rm TSO}\right]S_0+ dh\left(\sigma_{\rm add}^2-\sigma^2\right) }{dh\lambda_J\sin^2\phi} }.
\end{align}

Initial one-step matching is useful for calibrating the early exploration scale, but it does not automatically imply long-time spread which aids the exploration. One other way to make the fair comparison is by matching stationary variance.  Stationary variance depends on the balance between contraction and noise injection. For additive CBO, the frozen-consensus spread recursion is
\begin{align}
    S_{k+1}^{\rm add} =
    (1-\lambda_{\rm add} h)^2S_k^{\rm add} +
    d\sigma_{\rm add}^2h,
\end{align}    
If \((1-\lambda_{\rm add} h)^2<1\), the stationary spread is
\begin{align}
    S_\infty^{\rm add}
 = \frac{d\sigma_{\rm add}^2h}{ 1-(1-\lambda_{\rm add} h)^2 } = \frac{ d\sigma_{\rm add}^2}{        2\lambda_{\rm add}-\lambda_{\rm add}^2h  } .
\end{align}
For TSO, the corresponding recursion is $  S_{k+1}^{\rm TSO} =
    A^{\rm TSO}S_k^{\rm TSO}
    +  dh\left[
        \sigma^2
        + \lambda_J\sigma_J^2\sin^2\phi
    \right]$. If \(A^{\rm TSO}<1\), the stationary spread is
\[
    S_\infty^{\rm TSO}
    =
    \frac{dh\left[\sigma^2
+\lambda_J\sigma_J^2\sin^2\phi
        \right]}{1-A^{\rm TSO}}.
\]
Simplifying \(1-A^{\rm TSO}\), this can be written as
\begin{equation}
    S_\infty^{\rm TSO}
    =
    \frac{
        d
        \left[
            \sigma^2 +            \lambda_J\sigma_J^2\sin^2\phi
        \right]
    }{  2\eta+\lambda_J\sin^2\phi - h\left[
            \eta+\lambda_J(1-\cos\phi)
        \right]^2 }.
\end{equation}
Matching the stationary spread of additive CBO and TSO requires $ S_\infty^{\rm add}
    =
    S_\infty^{\rm TSO}$.
Denoting $D_{\rm TSO}
    := 2\eta
    + \lambda_J \sin^2\phi
    - h\left[\eta+ \lambda_J(1-\cos\phi)
    \right]^2$, the stationary-variance matching condition gives
\begin{equation}\label{tso_eqn_6.14}
    \sigma_J
    = \sqrt{\frac{\dfrac{\sigma_{\rm add}^2 D_{\rm TSO}}
            {2\lambda_{\rm add}-\lambda^2_{\rm add} h}-\sigma^2}{\lambda_J \sin^2\phi}}.
\end{equation}

\subsubsection{Experiment results}
 For each objective and each method, we perform $100$ independent trials with ensemble sizes $N=50$ and $N=100$. The particles are initialized from an isotropic Gaussian distribution centered in a corner region chosen to be far from the global minimizer. The initialization is normalized so that
$$
S_0=\mathbb{E}\lvert X_0-\mu_0\rvert^2=1,
$$
which, in dimension $d=2$, corresponds to covariance $\tfrac{1}{2} I_2$. All methods use the time horizon $T=100$, step size $h=0.1$, weight parameter $\alpha=50$. For additive CBO, we take $\lambda_{\mathrm{add}}=1$ (i.e. drift parameter) and $
\sigma_{\mathrm{add}}=1.$ 
The anisotropic CBO diffusion is selected by matching its initial one-step mean-square spread to that of additive CBO, giving (see \eqref{tso_eqn_6.9})
$$
\lambda_{\mathrm{aniso}}=1\; (\text{i.e. drift parameter}),
\qquad
\sigma_{\mathrm{aniso}}
=\sqrt{\frac{d\sigma_{\mathrm{add}}^2}{S_0}}
=\sqrt{2}.
$$
For TSO, the parameters are
$$
\eta=1,\qquad
\sigma=0.25,\qquad
\phi=0.5,\qquad
\lambda_J=1,\qquad
\beta=0.5,\qquad
\lambda_Y=1,\qquad
\kappa_1=\kappa_2=1.
$$
The jump amplitude is chosen by matching the stationary spread of TSO to that of additive CBO using \eqref{tso_eqn_6.14}. 
A run is classified as successful when the returned estimate lies within $\Delta_x=0.1$ of the global minimizer $x_\ast$. 

The results are reported in Tables~\ref{tab:distance_success_sigma1_gaussian} and~\ref{tab:success_rates_N100}. For both CBO variants, the returned estimate is the final weighted consensus point. For TSO, success is recorded when the final hunter position satisfies the distance criterion. Across all four objectives and both ensemble sizes, TSO achieves the highest success rate. Its mean success rate increases from $71.75\%$ for $N=50$ to $77\%$ for $N=100$, compared with $20.25\%$ and $29\%$, respectively, for additive CBO.  These results indicate that the improvement is consistent across the tested objectives under same exploratory energy.

\begin{table}[H]
    \centering
    \caption{Success rates over $100$ independent runs with $N = 50$ explorers.}
    \label{tab:distance_success_sigma1_gaussian}
    \begin{tabular}{l c c c}
        \toprule
        Objective function & Additive CBO & Anisotropic CBO & TSO \\
        \midrule
        Charged particle in an EM field & $26$ & $1$ & $77$ \\
        Double pendulum                      & $7$  & $0$ & $25$ \\
        Charged particle in non-uni. magnetic field         & $42$ & $0$ & $90$ \\
        Damped particle in a multi-well potential           & $6$  & $1$ & $95$ \\
        \bottomrule
    \end{tabular}
\end{table}

\begin{table}[H]
    \centering
    \caption{Success rates over $100$ independent runs with $N=100$ explorers.}
    \label{tab:success_rates_N100}
    \begin{tabular}{l c c c}
        \toprule
        Objective function & Additive CBO & Anisotropic CBO & TSO \\
        \midrule
        Charged particle in an EM field                 & $42$ & $0$ & $99$ \\
        Double pendulum                                 & $4$  & $0$ & $21$ \\
        Charged particle in a non-uni. magnetic field & $55$ & $0$ & $93$ \\
        Damped particle in a multi-well potential       & $15$ & $0$ & $95$ \\
        \bottomrule
    \end{tabular}
\end{table}

\subsection{Low-dimensional Bayesian inverse problem : comparison between CBS, EKS and TSO + Kalman calibration}
\label{tso_subsec:low_dimensional_experiment}

We compare consensus-based sampling \cite{carrillo2022consensus_sampling}, the ensemble Kalman sampler \cite{garbunoinigo2019interacting}, and the proposed two-stage TSO with Kalman calibration (see Section~\ref{tso_sec_uq}) on a nonlinear Bayesian inverse problem. We compare the methods under an identical forward-model evaluation budget, since evaluations of the forward map are assumed to constitute the dominant computational cost. We borrow this example from \cite{garbunoinigo2019interacting} (see also \cite{herty2018kinetic,carrillo2022consensus_sampling}). Here, we provide details for the sake of completeness. 

Let $x=(x_1,x_2)^\top\in\mathbb{R}^2$ denote the unknown parameter. We consider the one-dimensional boundary-value problem
\begin{equation*}
    -\frac{\mathrm{d}}{\mathrm{d}v}
    \left(\exp(x_1)\frac{\mathrm{d}p}{\mathrm{d}v}
    \right) =1, \qquad v\in[0,1],
    \label{tsoeqn:low_dim_pde}
\end{equation*}
subject to
\begin{equation*}
    p(0)=0,  \qquad
    p(1)=x_2.
\end{equation*}
The explicit solution is available:
\begin{equation*}
    p(v;x) =x_2v
    +  \exp(-x_1)\left(-\frac{v^2}{2}+\frac{v}{2}\right).\label{tso_eqn_low_dim_explicit_solution}
\end{equation*}
The forward operator consists of observations of $p$ at two locations, $\mathcal{G}(x)
    =  (p(0.25;x), p(0.75;x))$,
and the observed data are $y= (27.5,    79.7)$. We assume additive Gaussian observational noise, i.e. $y=\mathcal{G}(x)+\zeta$, $\zeta\sim\mathcal{N}(0,\Gamma)$, and $\Gamma=0.1^2 I_2$,
with the Gaussian prior $ x\sim\mathcal{N}(0,\Gamma_0)$, $\Gamma_0=10^2 I_2$. 
The posterior density is therefore  $   \pi(x\mid y)
    \propto
    \exp\big(-\Phi(x)\big)$,
where the negative log-posterior, up to an additive constant, is
\begin{equation}
    \Phi(x) =
    \frac{1}{2}
    \left| \mathcal{G}(x)-y
    \right|_{\Gamma}^{2}
    +
    \frac{1}{2}
    \left|x\right|_{\Gamma_0}^{2}.
    \label{eq:negative_log_posterior}
\end{equation}
Here, $|v|_A^2=v^\top A^{-1}v$. 
The reference posterior mean and covariance (computed via quadrature) are
\begin{equation*}
    m_{\mathrm{ref}}=
    \begin{pmatrix}-2.7139\\
        104.3458
    \end{pmatrix},
    \qquad
    C_{\mathrm{ref}}
    =
    \begin{pmatrix}  0.01291 & 0.02882\\
        0.02882 & 0.08078
    \end{pmatrix}.
    \label{tso_eqn:reference_moments}
\end{equation*}
The reference marginal $95\%$ credible intervals are
\begin{equation}
    x_1\in[-2.9194,-2.4740],
    \qquad
    x_2\in[103.7866,104.9008].
    \label{tso_eqn_refere_intervals}
\end{equation}
All ensemble methods have $N = 1000$ particles and the same initial ensemble $x_1^{(j)}\sim\mathcal{N}(0,1)$, $    x_2^{(j)}\sim\mathcal{U}(L,U)$, $  j=1,\ldots,N$.
This initialization follows the numerical setup commonly used for this test problem and is distinct from the Bayesian prior entering \eqref{eq:negative_log_posterior}. In particular, the prior remains $\mathcal{N}(0,10^2I_2)$ when evaluating the posterior objective. 

We compare with two other methods CBS \cite{carrillo2022consensus_sampling} and EKS \cite{garbunoinigo2019interacting}. For CBS, the particle update \cite[Section~2.5]{carrillo2022consensus_sampling} is
\begin{equation}
    X_{k+1}^{j, N}
    = \m_{k}^{\alpha} +
    \beta( X_k^{j, N}-\m_{k}^\alpha)
    +
    \sqrt{(1-\beta^2)(1+\alpha)}\,
    C_{k}^{1/2}\xi_k^{j},
    \label{eq:cbs_update}
\end{equation}
where $
    C_{k} = \sum_{j=1}^{N}
    w_k^j(X_k^{j,N}-\m_{k}^\alpha)
    (X_k^{j,N}-\m^{\alpha}_k)^\top$ with $
    w_k^{j}= \frac{\exp(-\alpha\Phi(X_k^{j,N}))}{\sum_{l=1}^{N}
        \exp(-\alpha\Phi(X_k^{l,N}))}$. As in \cite{carrillo2022consensus_sampling}, we use $\alpha=0.5$ and $\beta=0.5$ and perform $30$ updates. 

The set-up for implementation of ensemble Kalman sampler is from \cite{garbunoinigo2019interacting} based on linearly implicit split-step discretization of the derivative-free interacting particle system:
\begin{align*}
    X_{k+1}^{i,N} &=
    \hat{X}_{k+1}^{i, N} +
    \sqrt{2 h_k}\,
    C_k^{1/2}\xi_k^{i},
    \qquad
    \xi_k^{i}\sim\mathcal{N}(0,I_2), \\
    (I+ h_k C_k\Gamma_0^{-1})\hat{X}_{k+1}^{i,N} & = X_k^{i,N} - h_k D_k,
\end{align*}
where  $C_k = \frac{1}{N}\sum_{j=1}^{N}
    ( X_k^{j,N}-\bar{X}_k)
    (X_k^{j,N}-\bar{X}_k)^\top$ and $D_k =
    \frac{1}{N}
    \sum_{j=1}^{N}\langle \mathcal{G}(X_k^{j,N})-
        \bar{\mathcal{G}}_k,\mathcal{G}(X_k^{j,N})-y\rangle_{\Gamma}(X_k^{j,N}-\bar{X}_k)$ 
with  $\bar{X}_k =\frac{1}{N}\sum_{j=1}^{N}X_k^{j,N}$, $\bar{\mathcal{G}}_k = \frac{1}{N}\sum_{i=1}^{N}\mathcal{G}(X_k^{i,N})$
and adaptive time stepping $ h_k ={1}/(|D_k|+10^{-8})$.

 We compare CBS and EKS with TSO + Kalman calibration. A TSO particle system is first used to identify the region of high posterior probability. A few steps of deterministic Kalman calibration are then applied to construct an approximation of the local posterior distribution.
The parameters used in the TSO stage are
$h=0.4$, $\alpha_{\max}=30$, $\eta=1$ (we take $\eta_2 = 0$ and relabel $\eta_1$ as $\eta)$, $\sigma =0.8$, $\lambda_J =0.5$, $\sigma_J=2$, $\phi =\frac{\pi}{6}$, $\lambda_Y =10$, $
    \kappa_1 =1 , \kappa_2 = 1$.
 At the end of TSO steps, let $C_{\mathrm{TSO}}$ denote the ensemble covariance of the ensemble of TSO explorers, i.e.  
\begin{equation}
    C_{\mathrm{TSO}} =
    \frac{1}{N}
    \sum_{j=1}^{N}
    ( X^{j,N}
        - \bar{X})(X^{j,N} -
        \bar{X})^\top.
\end{equation}
We initialize as a Gaussian ensemble with mean equal to treasure hunter's position $Y_{\mathrm{TSO}}$ and empirical covariance exactly equal to $C_{\mathrm{TSO}}$. This construction uses the scale, anisotropy, and correlation learned by the explorers, while avoiding the transfer of any other incidental non-Gaussian structure. After that we employ $L_K$ steps of Kalman calibration described in \eqref{tsoeqn_kalman_gain}-\eqref{tsoeq:kalman_ensemble}.  
CBS and EKS use one initial and $30$ updated ensemble evaluations, therefore total $(1+30)N$ evaluations. For TSO + Kalman calibration, the localization stage uses  $(1+ n_{\mathrm{TSO}})N$ evaluations
while the calibration stage uses $(1+ L_K)N$. We ensure that all three methods use exactly the same forward-model evaluation budget. In this experiment, we take $L_K = 5$ and $n_{\mathrm{TSO}} = 24$.

TSO performs derivative-free optimization  while the Kalman updates are used only after the ensemble has entered a region where local regression is more credible.
A computational advantage is that TSO  does not require matrix operations and they are introduced  during the Kalman calibration steps. By contrast, consensus-based sampling forms and factorizes a weighted covariance at each update, while the ensemble Kalman sampler constructs an empirical covariance and solves an implicit linear system. The dense linear algebra computations may scale as $\mathcal{O}(Nd^2)$ or $\mathcal{O}(d^3)$. Consequently, restricting such operations to a short calibration stage may be beneficial as the parameter dimension increases. This advantage is conditional on the cost structure of the application, since the forward model would still dominate the total runtime.

We first compare the three methods when ensemble is initialized as $\mathcal{N}(0,1)\otimes \mathcal{U}(90,110)$ (as in \cite{carrillo2022consensus_sampling}). Consensus-based sampling gives the most accurate approximation in this experiment. The ensemble Kalman sampler accurately locates the posterior mean but produces a visibly overdispersed ensemble. This overdispersion is also reflected in the $x_2$ credible interval, $[103.6375,104.9603]$
 which is wider than the reference interval.  In this test-case, the conclusion is CBS $>$ TSO + Kalman calibration $>$ EKS as can be seen in Table~\ref{tsotab_close_init_summaries} and Table~\ref{tso_tab_close_init_metrics} as well as Figures~\ref{tso_fig_close_init_particle}-\ref{tso_fig_close_init_marginal}. 
 
 However, once we change the initialization of ensemble to $\mathcal{N}(0,1)\otimes\, \mathcal{U}(60,90)$, the CBS performance degrades drastically. The ensemble collapses in the region defined by initial distribution producing severely biased approximation ($\Phi(\text{mean}) = 7496.7$ against the reference value $54.5$). EKS is essentially unaffected by the change in initialization, since its covariance-preconditioned drift lets particles travel across the parameter space.  
 TSO + Kalman calibration is the most robust of the three: the explorer-hunter search locates the correct basin independently of initialization, and the resulting calibration attains the lowest relative covariance error of any method in this setting ($0.124$ vs. $0.964$ for CBS and vs. $0.487$ for EKS).  
 Note that we only change the initialization, all other involved parameters for the three methods are kept the same. In this test case, it is clear that TSO + Kalman calibration $>$ EKS $\gg$ CBS (see Table~\ref{tso_tab_far_init_summaries} and  Table~\ref{tso_tab_far_initia_metrics} as well as Figures~\ref{tso_fig_far_init_ensemble}-\ref{tso_fig_far_init_marginal}). 
 Also note that even before Kalman calibration, TSO hunter situates itself in the global basin as can be seen in Table~\ref{tso_tab_hunter_position_before_kalman}.

 \begin{table}[H]
\centering
\caption{TSO hunter position at the $24$-th step, i.e., before Kalman calibration.}
\label{tso_tab_hunter_position_before_kalman}
\scriptsize
\renewcommand{\arraystretch}{1.20}
\begin{tabular}{@{}lcc@{}}
\toprule
\textbf{Initialization}
&
\textbf{TSO hunter position} (before Kalman calibration)
&
$\boldsymbol{\Phi}$
\\
\midrule

$\mathcal{N}(0,1)\otimes\mathcal{U}(90,110)$
&
$(-2.7312,\;104.3288)$
&
$54.4927$
\\

$\mathcal{N}(0,1)\otimes\mathcal{U}(60,90)$
&
$(-2.7395,\;104.2913)$
&
$54.4958$
\\

\bottomrule
\end{tabular}
\end{table}

\begin{table}[H]
\centering
\caption{Summary of results with ensemble initialization $ \mathcal{N}(0, 1) \otimes \mathcal{U}(90, 110)$.
}\label{tsotab_close_init_summaries}
\scriptsize
\renewcommand{\arraystretch}{1.25}
\resizebox{\textwidth}{!}{
\begin{tabular}{@{}lccc@{}}
\toprule
\textbf{Reference}
&
\textbf{CBS}
&
\textbf{EKS}
&
\textbf{TSO + Kalman calibration}
\\
\midrule

$\small
\text{Mean}=
(-2.7139,\;104.3458)$
&
$\small
(-2.7167,\;104.3418)$
&
$\small
(-2.7184,\;104.3184)$
&
$\small
(-2.7053,\;104.3204)$
\\[1mm]

$\small
\text{Covariance}=
\begin{pmatrix}
0.0129 & 0.0288\\
0.0288 & 0.0808
\end{pmatrix}$
&
$\small
\begin{pmatrix}
0.0119 & 0.0275\\
0.0275 & 0.0811
\end{pmatrix}$
&
$\small
\begin{pmatrix}
0.0180 & 0.0415\\
0.0415 & 0.1212
\end{pmatrix}$
&
$\small
\begin{pmatrix}
0.0138 & 0.0259\\
0.0259 & 0.0703
\end{pmatrix}$
\\[3mm]

$\small
\text{Marginal variances}=
(0.0129,\;0.0808)$
&
$\small
(0.0119,\;0.0811)$
&
$\small
(0.0180,\;0.1212)$
&
$\small
(0.0138,\;0.0703)$
\\[1mm]

$\small
\begin{aligned}
x_1 &: [-2.9194,-2.4740],\\
x_2 &: [103.7866,104.9008]
\end{aligned}$
&
$\small
\begin{aligned}
x_1 &: [-2.9260,-2.5063],\\
x_2 &: [103.7642,104.8601]
\end{aligned}$
&
$\small
\begin{aligned}
x_1 &: [-2.9825,-2.4604],\\
x_2 &: [103.6375,104.9603]
\end{aligned}$
&
$\small
\begin{aligned}
x_1 &: [-2.9295,-2.4724],\\
x_2 &: [103.8158,104.8511]
\end{aligned}$
\\[3mm]

$\small
\Phi(\text{mean})=54.5115$
&
$\small
54.5054$
&
$\small
54.5300$
&
$\small
54.6301$
\\

\bottomrule
\end{tabular}
}
\end{table}

\begin{table}[H]
\centering
\caption{Error with ensemble initialization $x_1 \sim \mathcal{N}(0, 1), x_2\sim \mathcal{U}(90, 110)$.}
\label{tso_tab_close_init_metrics}
\scriptsize
\renewcommand{\arraystretch}{1.20}
\resizebox{\textwidth}{!}{
\begin{tabular}{@{}lccc@{}}
\toprule
\textbf{Metric}
&
\textbf{CBS}
&
\textbf{EKS}
&
\textbf{TSO + Kalman calibration}
\\
\midrule

Mean absolute error
&
$(0.0029,\;0.0039)$
&
$(0.0045,\;0.0273)$
&
$(0.0086,\;0.0253)$

\\

Covariance absolute error
&
$\begin{pmatrix}
0.0010&0.0013\\
0.0013&0.0003
\end{pmatrix}$
&
$\begin{pmatrix}
0.0051&0.0127\\
0.0127&0.0404
\end{pmatrix}$
&
$\begin{pmatrix}
0.0009&0.0030\\
0.0030&0.0105
\end{pmatrix}$
\\[3mm]

Relative covariance Frobenius error
&
$0.0234$
&
$0.4872$
&
$0.1240$
\\

Marginal Wasserstein distances
&
$(0.0078,\;0.0091)$
&
$(0.0163,\;0.0541)$
&
$(0.0107,\;0.0280)$
\\

\bottomrule
\end{tabular}
}
\end{table}


\begin{table}[H]
\centering
\caption{Summary of results with ensemble initialization $\mathcal{N}(0,1)\otimes \mathcal{U}(60, 90)$.}
\label{tso_tab_far_init_summaries}
\scriptsize
\renewcommand{\arraystretch}{1.25}
\resizebox{\textwidth}{!}{
\begin{tabular}{@{}lccc@{}}
\toprule
\textbf{Reference}
&
\textbf{CBS}
&
\textbf{EKS}
&
\textbf{TSO + Kalman calibration}
\\
\midrule

$\small
\text{Mean}=
(-2.7139,\;104.3458)$
&
$\small
(-2.2049,\;89.8390)$
&
$\small
(-2.7184,\;104.3184)$
&
$\small
(-2.7067,\;104.3165)$
\\[1mm]

$\small
\text{Covariance}=
\begin{pmatrix}
0.0129 & 0.0288\\
0.0288 & 0.0808
\end{pmatrix}$
&
$\small
\begin{pmatrix}
0.0013 & 0.0017\\
0.0017 & 0.0023
\end{pmatrix}$
&
$\small
\begin{pmatrix}
0.0180 & 0.0415\\
0.0415 & 0.1212
\end{pmatrix}$
&
$\small
\begin{pmatrix}
0.0137 & 0.0258\\
0.0258 & 0.0703
\end{pmatrix}$
\\[3mm]

$\small
\text{Marginal variances}=
(0.0129,\;0.0808)$
&
$\small
(0.0013,\;0.0023)$
&
$\small
(0.0180,\;0.1212)$
&
$\small
(0.0137,\;0.0703)$
\\[1mm]

$\small
\begin{aligned}
x_1 &: [-2.9194,-2.4740],\\
x_2 &: [103.7866,104.9008]
\end{aligned}$
&
$\small
\begin{aligned}
x_1 &: [-2.2772,-2.1334],\\
x_2 &: [89.7468,89.9305]
\end{aligned}$
&
$\small
\begin{aligned}
x_1 &: [-2.9825,-2.4604],\\
x_2 &: [103.6375,104.9604]
\end{aligned}$
&
$\small
\begin{aligned}
x_1 &: [-2.9305,-2.4744],\\
x_2 &: [103.8120,104.8473]
\end{aligned}$
\\[3mm]

$\small
\Phi(\text{mean})=54.5115$
&
$\small
7496.7148$
&
$\small
54.5300$
&
$\small
54.6299$
\\

\bottomrule
\end{tabular}
}
\end{table}

\begin{table}[ht]
\centering
\caption{Error with ensemble initialization $x_1 \sim \mathcal{N}(0, 1), x_2\sim \mathcal{U}(60, 90)$.}
\label{tso_tab_far_initia_metrics}
\scriptsize
\renewcommand{\arraystretch}{1.20}
\resizebox{\textwidth}{!}{
\begin{tabular}{@{}lccc@{}}
\toprule
\textbf{Metric}
&
\textbf{CBS}
&
\textbf{EKS}
&
\textbf{TSO + Kalman calibration}
\\
\midrule

Mean absolute error
&
$(0.5089,\;14.5067)$
&
$(0.0045,\;0.0273)$
&
$(0.0072,\;0.0293)$
\\
Covariance absolute error
&
$\begin{pmatrix}
0.0116&0.0271\\
0.0271&0.0785
\end{pmatrix}$
&
$\begin{pmatrix}
0.0051&0.0127\\
0.0127&0.0404
\end{pmatrix}$
&
$\begin{pmatrix}
0.0008&0.0030\\
0.0030&0.0105
\end{pmatrix}$
\\[3mm]

Relative covariance Frobenius error
&
$0.9640$
&
$0.4872$
&
$0.1242$
\\

Marginal Wasserstein distances
&
$(0.5089,\;14.5067)$
&
$(0.0162,\;0.0541)$
&
$(0.0096,\;0.0313)$
\\
\bottomrule
\end{tabular}
}
\end{table}

\begin{figure}[ht]
    \centering
    \includegraphics[width=1\linewidth]{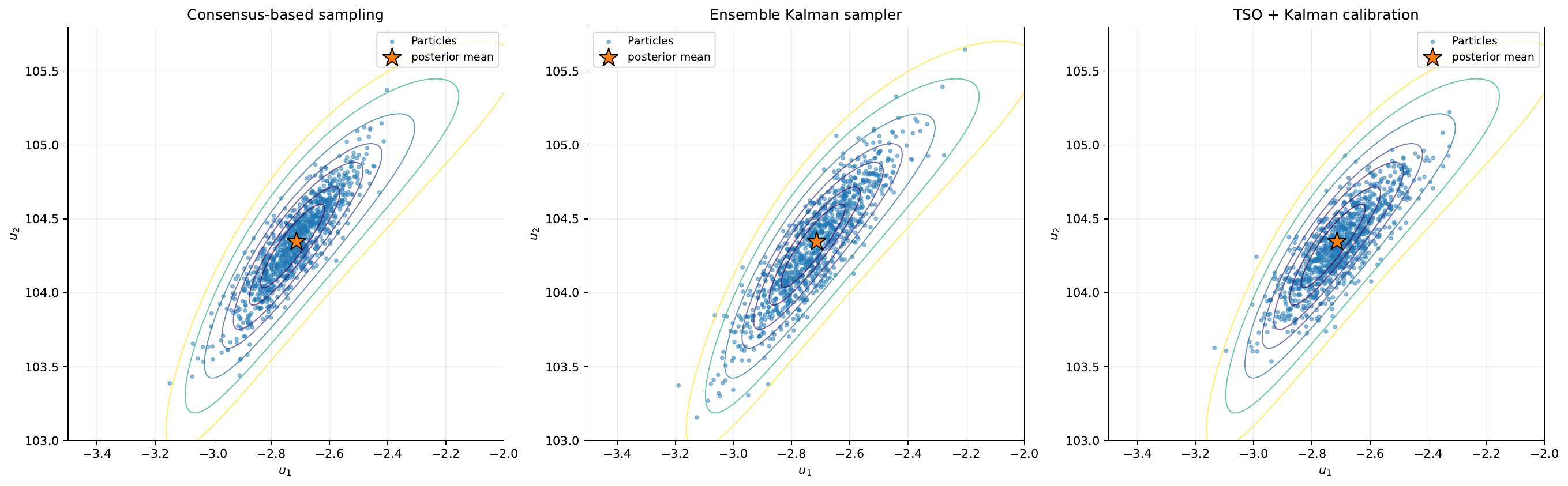}
    \caption{Illustration depicting particles at $30$-th (i.e. last) iteration for three methods : CBS, EKS and TSO + Kalman calibration where for all methods particles were initialized from $ \mathcal{N}(0,1)\otimes \mathcal{U}(90, 110)$.}
    \label{tso_fig_close_init_particle}
\end{figure}

\begin{figure}[H]
    \centering
    \includegraphics[width=1\linewidth]{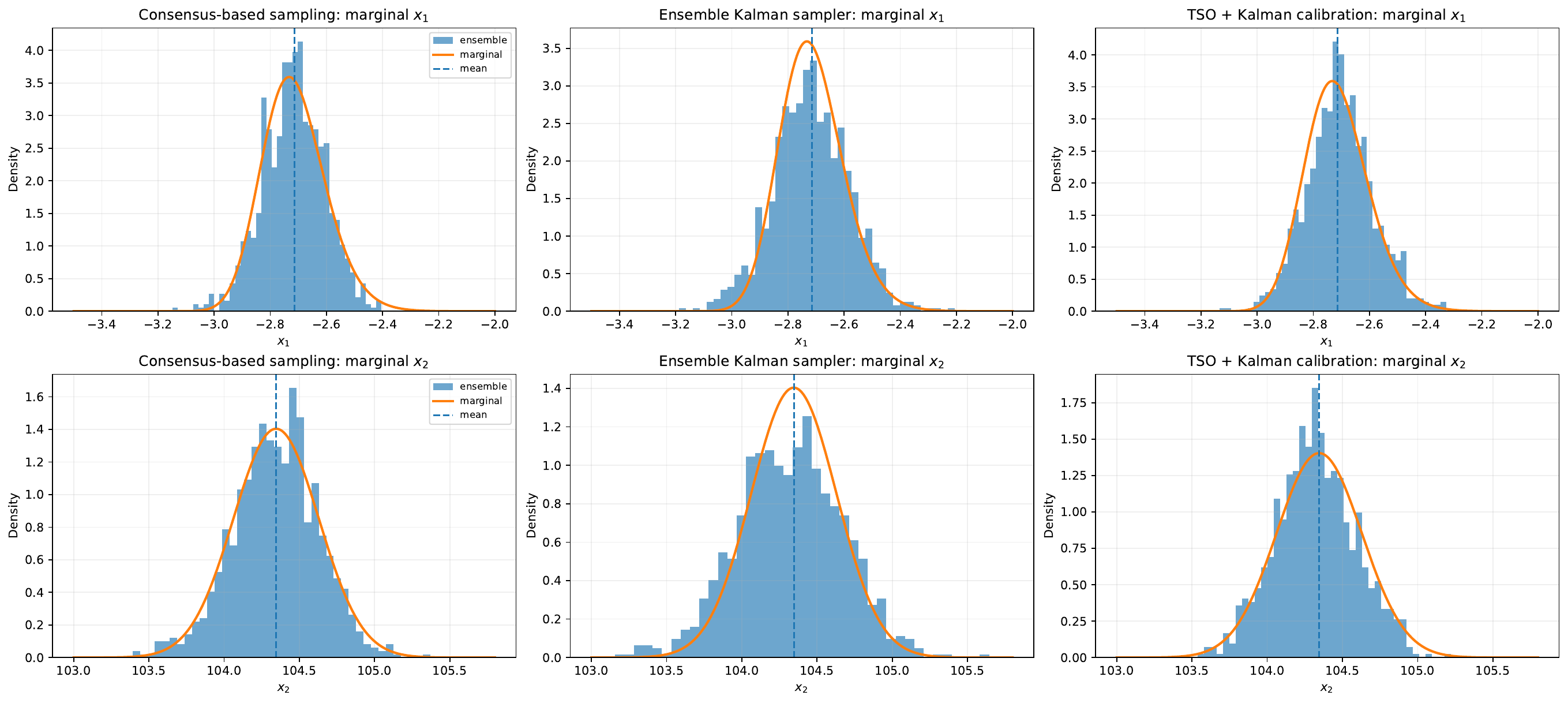}
    \caption{Marginal posterior approximations for $x_1$ and $x_2$ with initial ensemble
distribution $\mathcal{N}(0,1)\otimes\mathcal{U}(90,110)$.}
    \label{tso_fig_close_init_marginal}
\end{figure}

\begin{figure}[ht]
    \centering
    \includegraphics[width=1\linewidth]{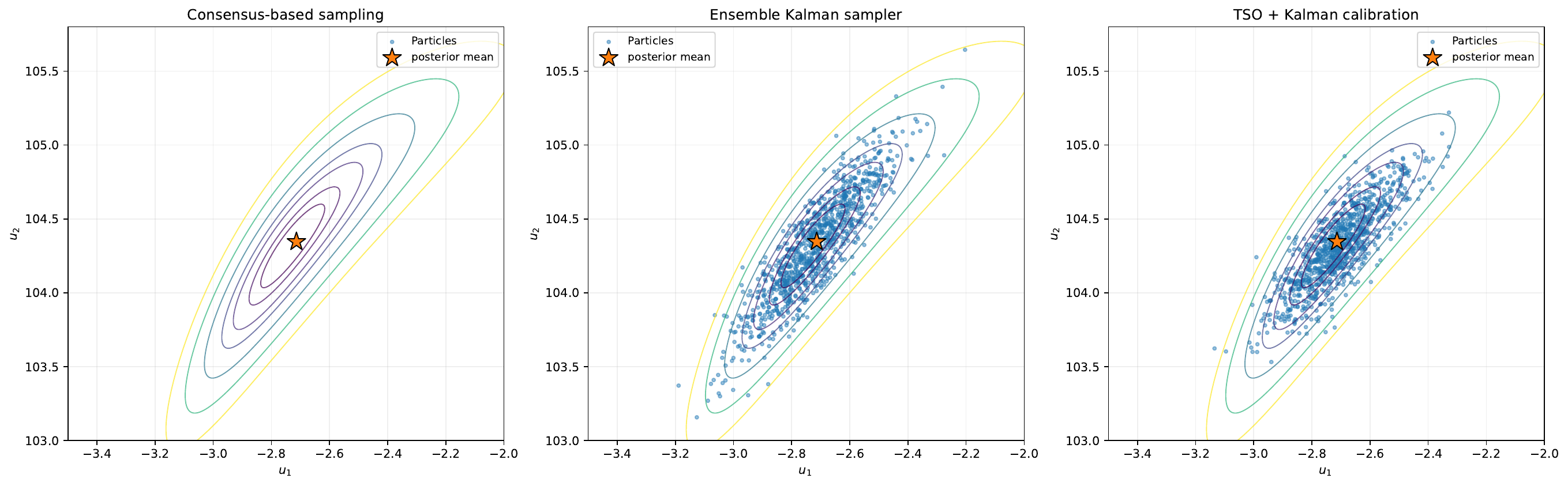}
    \caption{Illustration depicting particles at $30$-th (i.e. last) iteration for three methods : CBS, EKS and TSO + Kalman calibration where for all methods particles were initialized from $ \mathcal{N}(0,1)\otimes \mathcal{U}(60, 90)$.}
    \label{tso_fig_far_init_ensemble}
\end{figure}

\begin{figure}[H]
    \centering
    \includegraphics[width=1\linewidth]{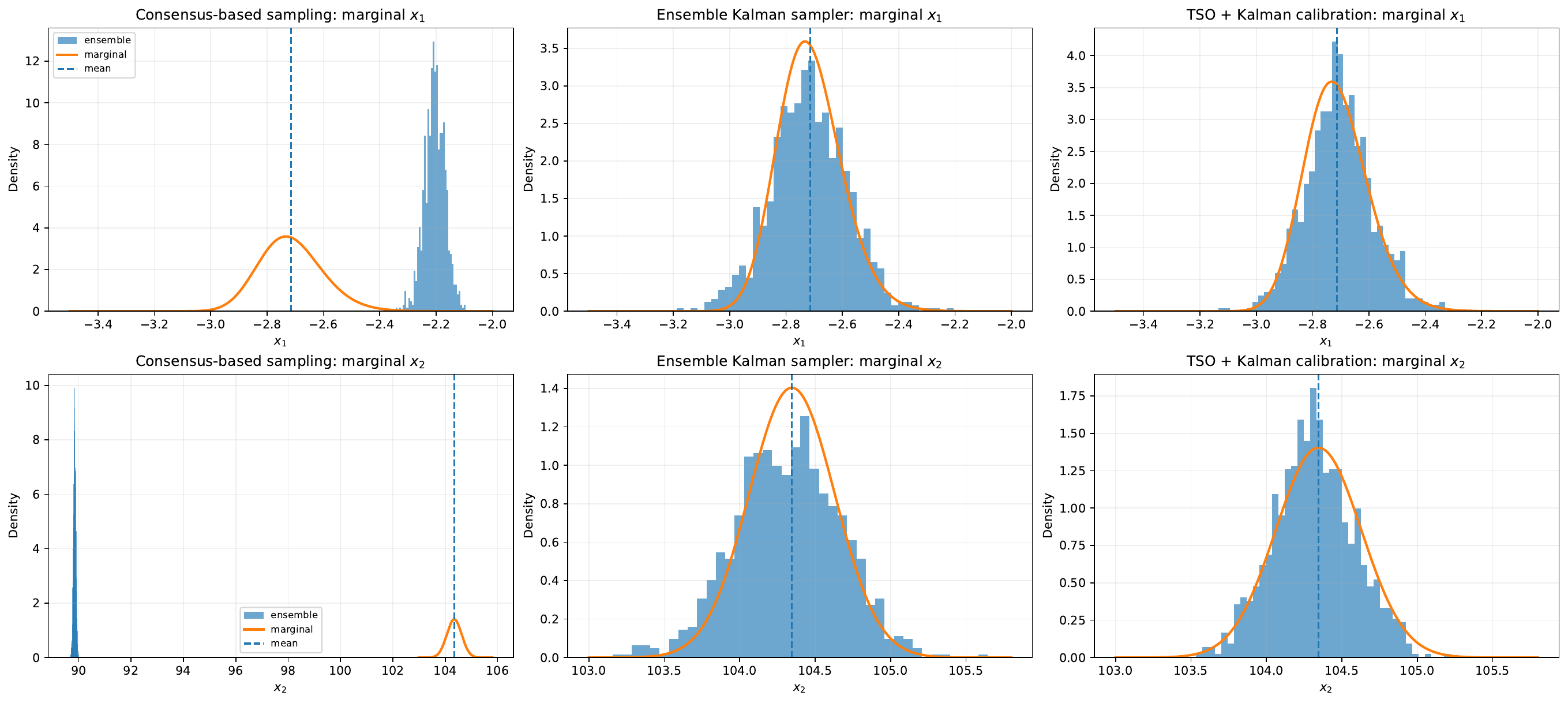}
    \caption{Marginal posterior approximations for $x_1$ and $x_2$ with initial ensemble
distribution $\mathcal{N}(0,1)\otimes\mathcal{U}(60,90)$.}
    \label{tso_fig_far_init_marginal}
\end{figure}

\section{Concluding Remarks}
We introduced Treasure Search Optimization, an interacting particle method that assigns exploration and exploitation to two distinct agents rather than balancing them within a single swarm.

On the analytical side, we established well-posedness of the mean-field limit which is a conditional McKean-Vlasov jump-diffusion SDE by constructing the fixed point pathwise along the hunter's jump times and selecting an adapted solution. We then characterized the steady state of the coupled dynamics and, through a quantitative Laplace approximation applied to the self-consistent equilibrium equation, proved that the hunter settles within $\mathcal{O}(1/\alpha)$ of the global optimizer. 

The conditional mean dynamics further revealed a macroscopic structure. The consensus drift coincides with the gradient of a Gaussian-smoothed free energy (see Figure~\ref{fig_from_eyes_of_swarm}), which clarifies why the swarm descends an effectively regularized landscape and remains insensitive to small spurious peaks that trap individual particles. This viewpoint is not specific to TSO, and offers a qualitative lens on consensus-based optimization and sampling more broadly. We also indicated how the equilibrium explorer cloud, whose covariance is fixed by algorithmic parameters, can be re-calibrated into a geometry-aware ensemble for uncertainty quantification in inverse problems. On the four ODE-constrained problems considered, TSO consistently outperformed additive and anisotropic consensus-based optimization with a modest choice of jump intensities and diffusion parameters. TSO + Kalman calibration also matched or exceeded CBS and EKS in accuracy while remaining robust to initialization. 

Several natural questions remain. The analysis in this paper is carried out at the level of the mean-field system. A quantitative conditional propagation of chaos result establishing convergence of the TSO explorers to the conditional mean-field limit, together with convergence to the equilibrium, is the subject of ongoing work. TSO is formulated as a hunter-explorer architecture for separating exploration and exploitation. In the present paper, this architecture is studied for a consensus-based explorer interaction. A natural direction is to treat the explorer drift $b$ as a design degree of freedom and explore different choices. A further extension is the multi-species setting, in which several explorer populations with distinct interaction mechanisms share a common hunter. We leave the study of such structures to future work.

\section*{Acknowledgment}
This work was supported by the Wallenberg AI, Autonomous Systems and Software Program (WASP) funded by the Knut and Alice Wallenberg Foundation. The computations/data handling were enabled by resources provided by the National Academic Infrastructure for Supercomputing in Sweden (NAISS) at Alvis at the Chalmers Centre for Computational Science and Engineering (C3SE), partially funded by the Swedish Research Council through grant agreement no. 2022-06725. I acknowledge the use of ChatGPT for help in code debugging and assistance in converting prototype code into the one used in numerical experiments. All outputs generated by the LLM were strictly manually reviewed and edited by me. I take full responsibility for the integrity of the code that it matches the methods, parameters, logical flow and and the resulting output as quoted in the paper.

\small
\bibliography{references}
\bibliographystyle{alpha}

\appendix

\section{Details of experiment of Section~\ref{comp_sec_cb_tso}}\label{detal_exp_appendix}
Below are the details of the ODE constrained inverse problem based optimization experiment of Section~\ref{comp_sec_cb_tso}. Here, we mention all the parameters involved which are used to construct the objective function \eqref{tso_eqn_new_6.5} based on generic ODEs \eqref{tso_eqn_new_6.4}. 
\begin{itemize}
    \item
Charged particle in a nonuniform electromagnetic field : The decision variable is the initial velocity $x=v_0=(v_{0,1},v_{0,2})$
and the initial position is $q(0)=(0,0)$. The state is $z=(q,v)\in\mathbb R^4$, and the dynamics is $\dot q=v$, $\dot v= \omega(q) (-v_2, v_1) + E(q)$, where $\omega(q) = \omega_0+\omega_1\sin(kq_1)\cos(kq_2)$, and $
    E(q) =
    E_{\rm strength}(\sin(kq_2), \cos(kq_1))$. 
The parameters are $t_f =8$, $x_{\rm ref}=(0.483,-1.433)$, $\omega_0=2.0$,
    $\omega_1=1.6$, $k=3.0$, $E_{\rm strength}=0.4$ 
and $\omega_v=0.2$. The search domain is $[-5,5]\times[-5,5]$. 

\item Double pendulum: The decision variable is the initial angular velocity $x = \omega_0 = (\omega_1(0), \omega_2(0))$. The initial angles are $(x_1(0), x_2(0)) = (1.1, -0.7)$. The state is $z = (x_1, x_2, \omega_1, \omega_2)$, and let $\delta = x_2 - x_1$. The dynamics are $\dot{x}_1 = \omega_1$, $\dot{x}_2 = \omega_2$,
\begin{align}
    \dot\omega_1 = \frac{m_2 l_1\omega_1^2\sin\delta\cos\delta + m_2g\sin x_2\cos\delta + m_2l_2\omega_2^2\sin\delta - (m_1+m_2)g\sin x_1}{(m_1+m_2)l_1-m_2l_1\cos^2\delta},
\end{align}
\begin{align}
    \dot\omega_2 = \frac{-m_2l_2\omega_2^2\sin\delta\cos\delta + (m_1+m_2)(g\sin x_1\cos\delta - l_1\omega_1^2\sin\delta - g\sin x_2)}{\frac{l_2}{l_1}((m_1+m_2)l_1-m_2l_1\cos^2\delta)}.
\end{align}
The parameters are $t_f = 8$, $x_{\rm ref} = (-0.4, -0.3)$, $m_1 = m_2 = 1$, $l_1 = l_2 = 1$, $g = 9.81$, and $\omega_v = 0.05$. The terminal angular error is computed using wrapped angle differences. The search domain is $[-10, 10] \times [-10, 10]$.

\item Charged particle in non-uni. magnetic field: The decision variable is the initial velocity $x = v_0 = (v_{0,1}, v_{0,2})$, with initial position $q(0) = (0.3, 0)$. The state is $z = (q, v)$, and the dynamics are $\dot{q} = v$ and
\begin{align}
    \dot{v} = B(q) \begin{pmatrix} -v_2 \\ v_1 \end{pmatrix} - c_{\rm conf}q,
\end{align}
where $B(q) = B_0 + B_1\cos(kq_1)\cos(kq_2)$. The parameters are $t_f = 10$, $x_{\rm ref} = (-0.067, 0.2)$, $B_0 = 3.0$, $B_1 = 2.0$, $k = 3.0$, $c_{\rm conf} = 0.4$, and $\omega_v = 0.2$. The search domain is $[-5, 5] \times [-5, 5]$.

\item Damped particle in multi-well potential : The decision variable is the initial momentum $x = p_0 = (p_{0,1}, p_{0,2})$, with initial position $q(0) = (-1, 0)$. The dynamics are $\dot{q} = p$ and $\dot{p} = -\nabla V(q) - \gamma p$, where $\gamma = 0.2$ and
\begin{align}
    \nabla V(q) = \begin{pmatrix} 4q_1(q_1^2-1) + 1.2\cos(4q_1)\cos(3q_2) \\ q_2 - 0.9\sin(4q_1)\sin(3q_2) \end{pmatrix}.
\end{align}
The parameters are $t_f = 8$, $x_{\rm ref} = (-1.86, 0)$, and $\omega_v = 0.1$. The search domain is $[-6, 6] \times [-6, 6]$.

\end{itemize}
\end{document}